\documentclass[a4j,12pt,leqno]{amsart}
\usepackage[foot]{amsaddr} 
\usepackage{comment}
\usepackage{color}
\usepackage{url}
\usepackage{amsmath,amssymb}
\usepackage[all,cmtip]{xy}
\usepackage{shuffle}
\usepackage{mathrsfs}
\usepackage{a4wide}
\makeatletter
\renewcommand{\section}{\@startsection
	{section}{1}{0mm}{5mm}{2mm}{\raggedright\bfseries}}

\@addtoreset{equation}{section}
\makeatother

\theoremstyle{plain} 
\newtheorem{theorem}{\indent\sc Theorem}[section] 
\newtheorem{lemma}[theorem]{\indent\sc Lemma}
\newtheorem{corollary}[theorem]{\indent\sc Corollary}
\newtheorem{proposition}[theorem]{\indent\sc Proposition}

\newtheorem{conjecture}[theorem]{\indent\sc Conjecture}
\newtheorem{hypothesis}[theorem]{\indent\sc Hypothesis}

\theoremstyle{definition} 
\newtheorem{definition}[theorem]{\indent\sc Definition}
\newtheorem{remark}[theorem]{\indent\sc Remark}
\newtheorem{example}[theorem]{\indent\sc Example}

\def\rmB{{\mathrm B}}

\def\rmE{{\mathrm E}}

\def\rmH{{\mathrm H}}

\def\rmR{{\mathrm R}}

\def\bfC{{\mathbf C}}

\def\bfF{{\mathbf F}}
\def\bfG{{\mathbf G}}

\def\bfP{{\mathbf P}}
\def\bfQ{{\mathbf Q}}
\def\bfR{{\mathbf R}}

\def\bfZ{{\mathbf Z}}

\def\calB{{\mathcal B}}
\def\calC{{\mathcal C}}

\def\calF{{\mathcal F}}
\def\calG{{\mathcal G}}

\def\calK{{\mathcal K}}

\def\calO{{\mathcal O}}

\def\calU{{\mathcal U}}
\def\calV{{\mathcal V}}

\def\scrE{{\mathscr E}}

\def\scrG{{\mathscr G}}

\def\scrM{{\mathscr M}}

\def\frakd{{\mathfrak d}}

\def\frakt{{\mathfrak t}}
\def\fraku{{\mathfrak u}}

\def\frakH{{\mathfrak H}}

\def\be{\boldsymbol e}

\def\bk{\boldsymbol k}

\def\bo{\boldsymbol o}
\def\bp{\boldsymbol p}
\def\bq{\boldsymbol q}

\def\bu{\boldsymbol u}
\def\bv{\boldsymbol v}

\def\bz{\boldsymbol z}

\def\B+{\mathop{{\mathbf B}^+_{\mathbf{dR}}}\nolimits}

\def\D+{\mathop{D^0_{\mathbf{dR}}}\nolimits}

\def\RepCrys_F{\mathop{\mathbf{Rep}_{\bfQ_\ell}^{\mathbf{crys}}(G_F)}\nolimits}

\def\Isom{\mathrm{Isom}}
\def\End{\mathrm{End}}
\def\Aut{\mathrm{Aut}}
\def\Hom{\mathrm{Hom}}

\def\Ext{\mathrm{Ext}}
\def\Rep{{\mathrm{Rep}}}

\def\Vec{\mathop{\mathrm{Vec}}\nolimits}

\def\Gr{\mathrm{Gr}}
\def\Sym{\mathrm{Sym}}

\def\GL{\mathbf{GL}}
\def\SL{\mathbf{SL}}

\def\Lie{\mathrm{Lie}}

\def\std{\mathrm{std}}

\def\univ{\mathrm{univ}}

\def\dR{{\mathrm{dR}}}
\def\crys{{\mathrm{crys}}}
\def\reg{{\mathrm{reg}}}

\def\MT{\mathop{\mathsf{MT}}\nolimits}

\def\ME{\mathop{\mathsf{ME}}\nolimits}

\def\IRep_F{\mathop{\mathrm{Ind}\mathbf{Rep}_{\mathbf Q_p}^{\mathbf{crys}}(G_F)}\nolimits}

\def\pet{\mathop{{p\mathrm{\acute{e}t}}}\nolimits}

\def\cyc{\mathop{\mathrm{cyc}}\nolimits}

\def\spec{\mathrm{Spec}}

\def\Obj{{\mathrm{Obj}}}
\def\wtd{{\mathrm{wtd}}}

\def\o+{{\oplus}}

\def\bo+{{\bigoplus}}

\def\et{\mathop{\mathrm{\acute{e}t}}\nolimits}

\def\Obj{{\mathrm{Obj}}}

\def\unip{{\mathrm{un}}}
\def\la{\langle}
\def\ra{\rangle}
\def\pr{{\mathrm{pr}}}

\def\b0{{\boldsymbol 0}}

\def\ontomap{\twoheadrightarrow}
\def\intomap{\hookrightarrow}
\def\isom{\xrightarrow{\sim}}

\makeatother

\title[The depth-weight compatibility and the Bloch-Kato conjecture]
{The depth-weight compatibility on the motivic fundamental Lie algebra and the Bloch-Kato conjecture for modular forms}
\author{Kenji Sakugawa$^1$}
\address{$^1$Faculty of Education, Shinshu University, 6-Ro, Nishi-nagano, Nagano 380-8544, Japan.}
\email{sakugawa\_kenji@shinshu-u.ac.jp}
\subjclass[2020]{Primary 11M32; Secondary 11F11, 18M25}
\keywords{Mixed Tate motive, Multiple zeta value, Modular form, Selmer group}
\begin{document}
\maketitle
\begin{abstract}
	Let $p$ be a prime number and let $V$ be a continuous representation of $\mathrm{Gal}(\overline {\mathbf Q}/\mathbf Q)$
	on a finite dimensional $\mathbf Q_p$-vector space, which is geometric. One of the Bloch-Kato conjectures for $V$
	predicts that the rank of the Hasse-Weil $L$-function
	of $V$ at $s=0$ coincides with the rank
	of Bloch-Kato Selmer group of $V^\vee(1)$.
	In this paper, we prove that the depth-weight compatibility on the fundamental Lie algebra of the mixed Tate motives over $\mathbf Z$ implies the Bloch-Kato conjecture for
	the $p$-adic Galois representations associated with full-level Hecke eigen cuspforms.
\end{abstract}
\section{Introduction}\label{intro}
The relation between analytic and algebraic invariants of $p$-adic Galois representations
is one of the basic themes in Number theory.
The Bloch-Kato conjecture predicts
an explicit relation between the Hasse-Weil $L$-function and
the Bloch-Kato Selmer group of the given
$p$-adic Galois representation.
\begin{conjecture}[Bloch-Kato conjecture, {\cite[Conjecture 5.1.3]{BC06}}]\label{con0-1}
	Let $V$ be a geometric irreducible representation of $\mathrm{Gal}(\overline{\mathbf Q}/\mathbf Q)$
	over a $p$-adic field $\mathcal K$.
	Then, the Hasse-Weil $L$-function $L(V,s)$ of $V$ has a meromorphic continuation to $\mathbf C$ and the following equation
	holds$:$
	\[
	\mathrm{rank}_{s=0}L(V,s)=\dim_{\mathcal K}\mathrm H^1_{\mathrm f}(\mathbf Q,V^\vee(1))-\dim_{\mathcal K}\mathrm H^0(\mathbf Q,V^\vee(1)).
	\]
	Here, $V^\vee$ is the dual Galois representation to $V$
	and $\mathrm H^1_{\mathrm f}(\mathbf Q,V^\vee(1))$ is the Bloch-Kato Selmer group
	with coefficients in $V^\vee(1)$.
\end{conjecture}
When $V=V_p(E)$ for an elliptic curve $E$ over $\mathbf Q$,
this conjecture is equivalent to the weak BSD-conjecture under the finiteness of
the $p$-primary part of the Tate-Shafarevich group of $E$.
The purpose of this paper is to link multiple zeta values and the Bloch-Kato conjecture for modular forms,
in particular, to show an unexpected implication that ``depth$=$weight'' derives the conjectural equation above for Galois representations attached to full-level cuspforms.

A \emph{multiple zeta value} (MZV) is a real number expressed by a convergent multiple Dirichlet series
\[
\zeta(\bk)=\sum_{n_1>n_2>\cdots>n_d>0}\frac{1}{n_1^{k_1}n_2^{k_2}\cdots n_d^{k_d}},
\]
where $\bk=(k_1,\dots,k_d)\in \bfZ_{>0}^d,\ k_1>1$.
The 
number $d$ is called the \emph{depth} of $\bk$.
This  real number is a period of a mixed Tate motive over $\bfZ$, and conversely, by the famous theorem of Brown (\cite[Theorem 1.1]{Brown12}), any period of mixed Tate motives over $\bfZ$ can be written as a $\bfQ[(2\pi\sqrt{-1})^{-1}]$-linear
combination of MZVs.
Let $\MT(\mathbf Z)$ be the category of mixed Tate motives over $\mathbf Z$ and
let $\frakt_*$ denote the graded fundamental Lie algebra
of $\MT(\mathbf Z)$.
Brown's theorem implies that the graded dual of the universal enveloping algebra $\calU(\frakt_*)$
of $\frakt_*$ is naturally isomorphic to the space of motivic MZVs modulo motivic $\zeta(2)$.
Therefore, we can impose the \emph{depth filtration} $D^\bullet \frakt_*$ on it. See Definition \ref{dfn2-1} for a concrete description of $D^\bullet\frakt_*$.

In recent years, attempts have been made to extend $\MT(\bfZ)$ in order to gain a deeper understanding of $D^\bullet\frakt_*$.
Let $\scrM_{1,1}$ be the moduli stack of elliptic curves.
In \cite{HM15}, Hain and Matsumoto introduced the category $\mathsf{UMEM}$ of universal
mixed elliptic motives over $\scrM_{1,1}$ and studied basic properties of it.
Roughly speaking, this category is defined to be the minimal Tannakian subcategory of
systems of realizations over $\scrM_{1,1}$ that contains the universal elliptic curve and whose
objects have \emph{motivic} degenerate fibers at the cusp of $\scrM_{1,1}$.
For the sake of affinity with Galois representations, in this paper, we use the category $\ME_p(\scrM_{1,1})$ of the mixed elliptic smooth $\mathbf Q_p$-sheaves, which is a $p$-adic avatar of $\mathsf{UMEM}$, in stead of $\mathsf{UMEM}$.
This category contains $\MT_p(\bfZ)=\MT_p(\bfZ)\otimes\bfQ_p$ as a full-subcategory.
The specialization at the cusp defines the functor $\ME_p(\scrM_{1,1})\to \MT_p(\bfZ)$ of $\bfQ_p$-linear Tannakian categories. Then, $\frakt_*\otimes_{\bfQ}\bfQ_p$ can be embedded into the graded fundamental Lie algebra $\mathfrak g_{\ME_p}$ of $\ME_p(\scrM_{1,1})$.

In Section 3, we impose the \emph{global weight filtration} $W^\bullet\frakt_{*,p}$ on $\frakt_{*,p}:=\frakt_*\otimes_{\mathbf Q}\mathbf Q_p$ by using the embedding to $\mathfrak g_{\ME_p}$ and the natural weight filtration on $\mathfrak g_{\ME_p}$.
Our depth-weight compatibility hypothesis (=Hypothesis \ref{hyp2-2}) is that the
following equation holds for any negative integers $w,n$ such that $2|w$:
\begin{equation}
	\label{eq0}
	W_n\frakt_{w,p}=D^{w/2-n}\frakt_{w,p}.
\end{equation}
The following is our main result.
\begin{theorem}
	\label{thm0-2}Let $f$ be a Hecke eigen cuspform of full-level of weight $k$
	and let $V_{f,p}$ denote the associated $p$-adic Galois representation of weight $k-1$ constructed by Shimura.
	Then, under the depth-weight compatibility hypothesis,
	Conjecture $\ref{con0-1}$ is true for $V_{f,p}$. 
\end{theorem}
Our hypothesis is natural one.
Indeed, if an \emph{elliptic analogue} of Brown's theorem \cite{Brown12}
is true, then the depth-weight compatibility hypothesis is also true. See Remark \ref{remhyp}.

Note that the Galois representation $V_{f,p}$
is \emph{non-critical} because the Gamma factor of $L(V_{f,p},s)=L(f,s)$ has
a pole at $s=0$. There exists known results of the Bloch-Kato conjecture for
\emph{critical} twists of $V_{f,p}$.
In the celebrated paper \cite{Kat04}, by using the Euler system method,
Kato proved Conjecture \ref{con0-1}
for the critical twist $V_{f,p}(j)$ of $V_{f,p}$
when the $L$-value $L(f,j)$ does not vanish.
In particular, when $j$ is not the central critical point,
Conjecture \ref{con0-1} for $V_{f,p}(j)$ is true (\cite[Theorem 14.2 (1)]{Kat04}).
The case of weight two, Conjecture \ref{con0-1} for $V_{f,p}(1)$
can be proved by combining results of Gross-Zagier and Kolyvagin (\cite{GZ86}, \cite{Kol90}) for the analytic rank less than or equal to one
modular elliptic curves.
Certain converse theorems of Gross-Zagier-Kolyvagin were proved by Skinner, Zhang, and Burange-Tian (\cite{Skinner20}, \cite{Zhang}, \cite{BT}).
For the higher weight case, Bella\"{i}che-Chenevier gave lower bounds
for the Bloch-Kato Selmer group when the sign of $f$ is $-1$ under technical conditions containing an Arthur's conjecture (\cite[Theorem 9.1.2]{BC06}).

Though our proof is on the same line as that of \cite{Kat04},
which will be carried out by proving non-vanishing of specializations of Euler systems,
the method for proving non-vanishing is completely different.
Kato proved such a non-triviality by relating Euler systems and (critical) $L$-values
by his explicit reciprocity law.
In this paper, we will prove a non-vanishing result
by relating Euler systems and relations of the depth-graded motivic Lie algebra $\mathfrak d$
\emph{under} the depth-weight compatibility hypothesis.

The plan of this paper is as follows.
In Section \ref{cryscompl}, we review a theory of
mixed elliptic smooth $\mathbf Q_p$-sheaves briefly.
In Section \ref{depth}, we recall the depth-graded motivic Lie
algebra
and prove a certain ``Eisenstein congruence relation'' under the hypothesis.
Then, in Section \ref{injectivity}, we prove the non-vanishing of
cup products of Eisenstein classes under the hypothesis.
In the final section, we give a proof of Theorem \ref{thm0-2} by relating those cup products and
specializations of the Kato Euler system.
\subsection*{Acknowledgements}
This work was supported by JSPS KAKENHI Grant Number JP23K03069.
\subsection*{Notation}
Throughout the paper, we fix a prime number $p$.
For each field $k$, we fix its separable closure $\bar k$.
The absolute Galois group $\mathrm{Gal}(\bar k/k)$
is denoted by $\mathscr G_k$.
We take $\overline{\mathbf Q}$ as a subfield of $\mathbf C$.
The symbol $\Rep_{\bfQ_p}(\scrG_k)$ denotes the category of continuous representations of $\scrG_k$ on finite dimensional $\bfQ_p$-vector spaces.
For a scheme $S$, $\mathsf{Sm}_p(S)$ denotes the category of
smooth $\mathbf Q_p$-sheaves on $S_{\et}$.
In this paper, $\scrM_{1,1}$ denotes the moduli stack of elliptic curves
over $\bfQ$-schemes and let $\pi\colon \mathscr E\to \scrM_{1,1}$ be the universal elliptic curve.

Let $\frak g$ be a Lie algebra over a field $k$ and let $V$ be a $k$-vector space
equipped with the Lie homomorphisms $F\colon \frak g\to \End_k(V)$.
Then, for each $x\in \frak g$, the symbol $W^{x=0}$ denotes the kernel of $F(x)$.

For a finite dimensional $k$-vector space $V$, the symbol $\underline{\Aut}_k(V)$ denotes the algebraic group over $k$ defined by $\underline{\Aut}_k(V)(R)=\Aut_{R}(V\otimes_kR)$ for each $k$-algebra $R$.
\section{Mixed elliptic smooth $\mathbf Q_p$-sheaves over $\scrM_{1,1}$}\label{cryscompl}
Our main theorem will be proved by relating cup products of
Eisenstein classes and relations of the depth-graded
motivic Lie algebra.
To connect those two things, we need a platform,
the fundamental Lie algebra of the category of mixed elliptic smooth $\mathbf Q_p$-sheaves on $\scrM_{1,1}$.
In this section, we summarize basic properties  of $\ME_p(\scrM_{1,1})$, which is proved by studying weighted completions
of the \'etale fundamental group of $\scrM_{1,1}$. For readers' convenience, we give a quick review of this notion in Appendix.
See  \cite{HM03}, \cite{HM04}, and \cite{HM15} for more details.

Let $E_q$ be the Tate elliptic curve defined over $\mathbf Q(\!(q)\!)$ (\cite[(8.8)]{KM84}) and let
\[
v\colon \spec(\overline{\mathbf Q}\{\!\{q\}\!\})\to \spec(\mathbf Q(\!(q)\!))\to \scrM_{1,1} 
\]
be the geometric point over the classifying morphism of the elliptic
curve $E_q\times_{\spec(\mathbf Q(\!(q)\!))}\spec(\overline{\mathbf Q}\{\!\{q\}\!\})$ over the field $\overline{\mathbf Q}\{\!\{q\}\!\}$ of formal Puiseux series over $\overline{\mathbf Q}$.
Let $ \mathsf{Sm}_p(\scrM_{1,1}/\mathbf Q)$ denote the category of smooth $\mathbf Q_p$-sheaves over $\mathscr M_{1,1}$. Then, the pull-back by $v$ and the coefficient action of $\scrG_{\bfQ}$ on $\overline{\mathbf Q}\{\!\{q\}\!\}$ defines a functor
\[
\bv\colon \mathsf{Sm}_p(\scrM_{1,1}/\mathbf Q)\to \mathsf{Sm}_p(\spec(\mathbf Q))={\Rep}_{\mathbf Q_p}(\mathscr G_{\mathbf Q})
\]
of neutral $\mathbf Q_p$-linear Tannakian categories.
\subsection{The definition}
For each $n\in \mathbf Z_{\geq 0}$, we define an object $\calV^n$
of $\mathsf{Sm}_p(\scrM_{1,1}/\mathbf Q)$ by
\[
{\calV}^n=\mathrm{Sym}^n\left(\rmR^1\pi_*\mathbf Q_p\right),
\]
where $\pi\colon \scrE\to \scrM_{1,1}$ is the universal elliptic curve.
\begin{definition}
	\label{dfn1.1}A \emph{mixed elliptic smooth $\mathbf Q_p$-sheaf} on $\scrM_{1,1}$ is a pair $(\mathcal F,W_\bullet\mathcal F)$ where
	\begin{itemize}
		\item $\mathcal F$ is an object of $\mathsf{Sm}_p(\scrM_{1,1})$ such that $\bv^*\mathcal F\in \Rep_{\mathbf Q_p}(\mathscr G_{\mathbf Q})$ is unramified outside $p$
		and crystalline at $p$,
		\item $W_\bullet\mathcal F $ is an increasing,
		saturated, and separated filtration of $\mathcal F$ such that
		\[
		{\Gr}^W_n\mathcal F\cong \bigoplus_{r\in \mathbf Z}{\calV}^{n+2r}(r)^{\oplus n_r}
		\]
		for some non-negative integers $n_r$.
	\end{itemize} 
	The filtration $W_\bullet\mathcal F$ is called the \emph{weight filtration} of $\calF$.
	A morphism $f\colon (\mathcal F,W_\bullet\mathcal F)\to (\mathcal G,W_\bullet\mathcal G)$
	is a morphism $f\colon \mathcal F\to \mathcal G$ in $\mathsf{Sm}_p(\scrM_{1,1}/\mathbf Q)$
	such that $f(W_n\calF)\subset W_n\calG$ for all $n$.
	The category of mixed elliptic smooth $\mathbf Q_p$
	is denoted by $\ME_p(\scrM_{1,1})$.
\end{definition}
According to Subsection \ref{generaltheory}, the category $\ME_p(\scrM_{1,1})$ can be regarded as an abelian full subcategory of $\mathsf{Sm}_p(\scrM_{1,1}/\mathbf Q)$ by the functor
\[
\ME_p(\scrM_{1,1})\to \mathsf{Sm}_p(\scrM_{1,1});\quad (\calF,W_\bullet\calF)\mapsto \calF.
\]
In particular, the weight filtration $W_\bullet \calF$ of $\calF$ is uniquely determined by $\calF$.
Therefore, we write $\calF$ for any object $(\mathcal F, W_\bullet\mathcal F)$ of $ \ME_p(\scrM_{1,1})$ for simplifying notation.

Any $p$-adic mixed Tate module over $\mathbf Z$ (see Appendix \ref{appendix2} for a precise definition)
can be regarded as an object $\mathcal F$ of $\ME_p(\scrM_{1,1})$ such that $\mathcal F|_{\scrM_{1,1}/\overline{\mathbf Q}}$
is constant. Let $\MT_p(\mathbf Z)$ denote the category of
$p$-adic mixed Tate modules over $\mathbf Z$.
By the remark above, $\MT_p(\mathbf Z)$ can be regarded as a full subcategory of $\ME_p(\scrM_{1,1})$.
Moreover, the functor $\bv^*$ defines a section
\begin{equation}
	\label{eq25}
{\ME}_p(\scrM_{1,1})\to {\MT}_p(\mathbf Z)
\end{equation}
of the natural functor ${\MT}_p(\mathbf Z)\to {\ME}_p(\scrM_{1,1})$. The functor (\ref{eq25}) is also denoted by $\bv^*$ by abuse of notation.
It is known that ${\MT}_p(\mathbf Z)$ has a natural $\mathbf Q$-structure.
	Let $\MT(\mathbf Z)$ be the category of mixed Tate motives over $\mathbf Z$.
See \cite[Section 1]{De-G05} for the precise definition of this category.
\begin{proposition}[{\cite[Corollary9.4]{HM04}}]
	\label{prop1-2}
	The $p$-adic \'etale realization functor induces an equivalence
	\[
	\MT(\mathbf Z)\otimes\mathbf Q_p\xrightarrow{\sim}{\MT}_p(\mathbf Z)
	\]
	of $\mathbf Q_p$-linear Tannakian categories.
\end{proposition}

\subsection{Weight filtrations of the fundamental Lie algebra}
Let $\ME_p^{\mathrm{ss}}(\scrM_{1,1})$ be the maximal semi-simple
Tannakian full subcategory of $\ME_p(\scrM_{1,1})$.
Since each simple object of $\ME_p(\scrM_{1,1})$ is isomorphic
to some $\calV^n(r)$, $\ME_p^{\mathrm{ss}}(\scrM_{1,1})$ has a tensor generator $\calV^1$. Therefore, the natural morphism
\[
\pi_1({\ME}_p^{\mathrm{ss}}(\scrM_{1,1}),v)\to \underline{\Aut}_{\mathbf Q_p}(\calV^1_v)\cong\mathrm{GL}_{2,\mathbf Q_p}
\]
is injective (Proof of \cite[Proposition 2.20 (b)]{DeM}).
Moreover, any simple objects of $\GL_{2,\bfQ_p}$-module is isomorphic to
$\mathrm{Sym}^n(\mathrm{std})\otimes\det^m$ for some $n,m\in \bfZ,\ n\geq 0$ where $\mathrm{std}$ is the standard representation of $\GL_{2,\bfQ_p}$,
the above homomorphism is also surjective by \cite[Proposition 2.21 (a)]{DeM}.
Therefore, we have a short exact sequence
\begin{equation}
	\label{eq1}
	1\to U\to \pi_1({\ME}_p(\scrM_{1,1}),v)\to \underline{\Aut}_{\mathbf Q_p}(\calV^1_v)\to 1
\end{equation}
of affine group schemes over $\mathbf Q_p$,
where $U$ is the prounipotent radical of
$\pi_1(\ME_p(\scrM_{1,1}),v)$.
Let
\[
\mathrm{pr}\colon \pi_1({\ME}_p(\scrM_{1,1}),v)\twoheadrightarrow \pi_1({\MT}_p(\mathbf Z),\omega_{\mathrm f})
\]
be the homomorphism of affine group schemes over $\bfQ_p$ induced by the natural fully faithful functor
\[
{\MT}_p(\mathbf Z)\hookrightarrow {\ME}_p(\scrM_{1,1}).
\]
Here $\omega_{\mathrm f}\colon \MT_p(\bfZ)\to \Vec_{\bfQ_p}$ is the forgetful functor.
Then, we define $\pi_1^{\mathrm g}({\ME}_p(\scrM_{1,1}),v)$ and $U^{\mathrm g}$ to be
the kernel of $\mathrm{pr}$ and the kernel of the homomorphism
\[
 U\hookrightarrow \pi_1({\ME}_p(\scrM_{1,1}),v) \ontomap \pi_1({\MT}_p(\mathbf Z),\omega_{\mathrm f}),
\]
respectively. We call $\pi_1^{\mathrm g}(\ME_p(\scrM_{1,1}),v)$ and $U^{\mathrm g}$ the geometric parts of $ \pi_1({\ME}_p(\scrM_{1,1}),v)$
and $U$, respectively.
Let $\mathfrak u$ and $\mathfrak u^{\mathrm g}$ denote the Lie algebras of $U$ and $U^{\mathrm g}$,
respectively.
For each $\calF\in \Obj(\ME_p(\scrM_{1,1}))$, its fiber $F:=\calF_v$ at $v$ has natural \emph{two} filtrations as follows:
\[
W_nF:=(W_n(\calF))_v,\quad W^\infty_n F:=W_n\bv^*\mathcal F.
\] 
Here, for $M_p\in\Obj(\MT_p(\bfZ))$, $W_* M_p$ denotes a unique $\scrG_{\bfQ}$-stable filtration such that $\Gr^W_n(M_p)$ is isomorphic to a direct sum of $\bfQ_p(-n/2)$ when $n$ is even, and is zero when $n$ is odd.
Then, we define two weight filtrations on
$\mathfrak u$ as follows.
\begin{definition}
	\label{dfn1-3}Let $\bv^*\colon \ME_p(\scrM_{1,1})\to \MT_p(\mathbf Z)$ be the same as before.
	\begin{itemize}
		\item[(1)]We define the \emph{global weight filtration} $W_\bullet\mathfrak u$
		by
		\[
		W_n\mathfrak u=\left\{u\in \mathfrak u\ \middle |\ u(W_i\mathcal F_v)\subset W_{n+i}\mathcal F_v ,\ \forall \mathcal F\in\Obj( {\ME}_p(\scrM_{1,1})),\ \forall i\in \mathbf Z\right\}.
		\]
		\item[(2)]We define the \emph{local weight filtration} $W^\infty_\bullet\mathfrak u$
		by
		\[
		W^\infty_n\mathfrak u=\left\{u\in \mathfrak u\ \middle |\ uW_i^\infty\mathcal F_v\subset W_{n+i}^\infty\mathcal F_v ,\ \forall \mathcal F\in \Obj({\ME}_p(\scrM_{1,1})),\ \forall i\in \mathbf Z\right\}.
		\]
	\end{itemize}
\end{definition}
Note that the local filtration is indexed by even integers.

Since the action of $\frak u$ on $\calF_v$ preserves $W_n(\calF_v)$ and is trivial on $\Gr^W_n(\calF_v)$ for any $\calF\in \Obj(\ME_p(\scrM_{1,1}))$ and any $n\in\bfZ$, $\mathfrak u$ has negative global weights. Namely, we have
\[
W_{-1}\mathfrak u=\mathfrak u.
\]
Moreover, this also has negative \emph{local} weights:
\[
W_{-2}^\infty\mathfrak u=\mathfrak u.
\]
Indeed, if we regard $\fraku$ as a $\scrG_{\bfQ}$-module by $\scrG_{\bfQ}\to \pi_1(\ME_p(\scrM_{1,1}),v)$ and by the adjoint action of $\pi_1(\ME_p(\scrM_{1,1}),v)$ on $\fraku$, $W_{-2}^\infty\fraku$ is a $\scrG_{\bfQ}$-stable Lie subalgebra of $\fraku$. Then, as the abelianization functor is right exact, we
have an exact sequence
\[
\rmH_1(W_{-2}^\infty\fraku)\to \rmH_1(\fraku)\to \rmH_1(\fraku/W_{-2}^\infty\fraku)\to 0
\]
of $\scrG_{\bfQ}$-modules. By the direct computation of $\rmH_1(\fraku)$ (\cite[Proposition 13.1]{HM15}, cf.\ Proposition \ref{propA0} (1)), the module $ \rmH_1(\fraku)$ is of negative weight.
Hence, we have that $\rmH_1(\fraku/W_{-2}^\infty\fraku)=0$ and this implies $\fraku=W_{-2}^\infty\fraku$.

Following an idea in \cite{HM15}, we fix a splitting of the two weight filtrations as follows.
According to Proposition \ref{propa3}, ${\Gr}^{W^\infty}_*{\Gr}^W_\bullet\circ v$ defines a fiber functor
on $\ME_p(\scrM_{1,1})$. Moreover, by Proposition \ref{propa4}, we have an isomorphism
\begin{equation}
	\label{fixedisom}
	v\cong {\Gr}^{W^\infty}_*{\Gr}^W_\bullet\circ v
\end{equation}
of fiber functors on $\ME_p(\scrM_{1,1})$.
Through out this paper, we fix an isomorphism above.
Then, for any $\calF\in \ME_p(\scrM_{1,1})$ with the fiber $F=\calF_v$, we equip $F$ with the canonical bigrading
\[
F=\bigoplus_{s,t\in\bfZ}F_{2s,t}
\]
such that
\[
W_n F=\bigoplus_{s,t\in \mathbf Z,\ t\leq n}F_{2s,t},\quad W^\infty_{2m} F=\bigoplus_{s,t\in \mathbf Z,\ s\leq m}F_{2s,t}.
\]
Put \[
V:=\calV_v^1=\bfQ_p\oplus \bfQ_p(-1).
\]
The fixed isomorphism (\ref{fixedisom}) induces a section
\begin{equation}
	\label{eq5}
	\mathrm{spl}\colon \underline{\Aut}_{\bfQ_p}(V)=\GL_{2,\bfQ_p}\hookrightarrow \pi_1({\ME}_p(\scrM_{1,1}),v)
\end{equation}
of $\pi_1({\ME}_p(\scrM_{1,1}),v)\to \GL_{2,\bfQ_p}$.
By using this splitting, we have the following equation:
\begin{equation}
	\label{eq6}
	F_{2m,n}=\left\{u\in F\ \middle |\ \rho_F\left(\mathrm{spl}\begin{pmatrix}
		t_1&0\\
		0&t_2
	\end{pmatrix}\right)u=t_1^{n-m}t_2^{m} u,\ \forall t_1,t_2\in \bfQ_p^\times\right\}.
\end{equation}
Here, $\rho_F\colon \pi_1(\ME_p(\scrM_{1,1}),v)\to\underline{\Aut}_{\bfQ_p}(F)$ is the canonical representation.
Note that our index is slightly different from that of \cite[Proof of Proposition B1]{HM15}
because they used the algebraic group $\underline{\Aut}_{\bfQ_p}(V^\vee)$ instead of $\underline{\Aut}_{\bfQ_p}(V)$.
\subsection{Generators of $\mathfrak u$}\label{gen}
For any positive integer $n$ greater than one, we put
\begin{equation}
	\label{eq7}
	\mathrm{E}_{2n}:=\mathrm H^1_{\et}\left(\scrM_{1,1}/\mathbf Q,\calV^{2n}(2n+1)\right),
\end{equation}
which is a one dimensional $\mathbf Q_p$-vector space.
By the fixed embedding $\overline{\mathbf Q}\subset\mathbf C$,
this space has a natural $\mathbf Q$-structure $\mathrm{E}_{2n,\mathbf Q}$
coming from the Betti-cohomology group of $\scrM_{1,1}(\mathbf C)$.
We denote by $\varphi_{2n}$ the $\mathbf Q$-basis
of $\mathrm{E}_{2n,\mathbf Q}$ which is sent by Betti-de Rham comparison map
to the class $\psi_{2n}$ associated with the normalized Eisenstein series of weight $2n+2$ (\cite[Subsection 9.1]{HM15}).
\begin{proposition}[{\cite[Proposition 13.1]{HM15}}]
	\label{prop1-5}There exists a natural isomorphism
	\[
	\mathrm H_1(\mathfrak u)\cong \prod_{m\geq 1}\mathrm H^1(\mathbf Q,\mathbf Q_p(2m+1))^\vee(2m+1)\times\prod_{n\geq 1}\mathrm{E}_{2n}^\vee\otimes_{\mathbf Q_p}\mathrm{Sym}^{2n}(V)(2n+1)
	\]
	of $\GL_{2,\bfQ_p}$-modules.
	Here, the superscript $\vee$ refers to the operation of taking the $\mathbf Q_p$-dual.
\end{proposition}
We define the element $\boldsymbol e_0$ of $\mathrm{Lie}(\GL_{2,\bfQ_p})$ by
\[
\boldsymbol e_0=\begin{pmatrix}
	0&1\\
	0&0
\end{pmatrix},
\]
where we fix a basis of $V=\bfQ_p\oplus \bfQ_p(-1)$ such that $e_0$ decreases the local weight by two.
Then, by the induced splitting \[
\mathrm{spl}'\colon \Lie(\GL_{2,\bfQ_p})\to \Lie(\pi_1(\ME_p(\scrM_{1,1}),v))
\]
by the fixed splitting $\mathrm{spl}$, $\boldsymbol e_0$ acts on $\mathfrak u$ by
\[
\mathrm{ad}(\be_0)(u):=[\mathrm{spl}'(\be_0),u]\quad \forall u\in\fraku.
\]
For each non-negative integer $i$, $\mathrm{ad}^i(\be_0)\in \End_{\bfQ_p}(\frak u)$ is defined to be the $i$th iteration of the endomorphism $\mathrm{ad}(\be_0)$ of $\fraku$.
Similar to \cite[Section 20]{HM15}, we choose lifts $\boldsymbol z_{2m+1}$
and $\boldsymbol e_{2n}$ of
the basis
$\mathrm H^1(\mathbf Q,\mathbf Q_p(2m+1))^\vee(2m+1)$ and $\mathrm{E}_{2n}^\vee(1)\subset \mathrm{E}_{2n}^\vee\otimes_{\mathbf Q_p}\mathrm{Sym}^{2n}(V)(2n+1)$
in
$\fraku_{-2(2m+1),-2(2m+1)}$ and in $\fraku_{-2,-2n-2}$,
respectively.
\begin{proposition}[{\cite[Section 20]{HM15}}]
	\label{prop1-6}
	\begin{itemize}
		\item[(1)]The set of elements
		\[
		\{\boldsymbol z_{2m+1},\quad \boldsymbol e_0^{i}\boldsymbol e_{2n+2}:=\mathrm{ad}^i(\boldsymbol e_0)\boldsymbol e_{2n+2}\ |\ m,n\geq 1,\ 0\leq i\leq 2n\}
		\]
		generates $\fraku$ topologically.
		\item[(2)]The space $\fraku_{-2,-2n-2}$ is one dimensional.
	\end{itemize}
\end{proposition}
By Proposition \ref{prop1-6} (2) and Proposition \ref{prop1-5}, the composition of homomorphisms
\[
\fraku_{-2,-2n-2}\intomap \fraku\ontomap \fraku^{\mathrm{ab}}=\rmH_1(\fraku)
\]
is injective. Therefore, it also follows from Proposition \ref{prop1-5} again that
$\boldsymbol e_{2n+2}$
is contained in the geometric part $\fraku^{\mathrm g}$ of $\fraku$, hence so is $\boldsymbol e_0^{i}\boldsymbol e_{2n+2}$.
We normalizes $\boldsymbol e_{2n+2}$ to be the dual to $\varphi_{2n+2}$.

For each graded piece $\fraku_{2r,m}$ of $\fraku$,
define $\fraku_{2r,m}^{e_0=0}$ by
\[
\fraku_{2r,m}^{e_0=0}=\left\{x\in \fraku_{2r,m}\ \middle|\ [\mathrm{spl}'(e_0),x]=0\right\}.
\]
Generally, for each representation \[
\rho_W\colon \Lie(\pi_1(\ME_p(\scrM_{1,1}),v))\to\End_{\bfQ_p}(W),
\]
of $\Lie(\pi_1(\ME_p(\scrM_{1,1}),v))$ on a $\bfQ_p$-vector space $W$, the symbol $W^{\be_0=0}$ denotes the kernel of the endomorphism $\rho_W(\mathrm{spl}'(\be_0))$ of $W$.
\begin{lemma}
	\label{lem1-7}The space $\fraku_{-2(2n+1),-2n-2}^{\boldsymbol e_0=0}$ 
	is one dimensional.
	In particular, we have\[
	{\fraku}_{-2(2n+1),-2n-2}^{\boldsymbol e_0=0}=
	\mathbf Q_p\boldsymbol e_0^{2n}\boldsymbol e_{2n+2}.
	\]	
\end{lemma}
\begin{proof}
	First, we note that there exists the following natural isomorphism:
	\[
	\mathrm{Sym}^{m}(V)^{\be_0=0}\cong\left(\bfQ(0)\oplus\bfQ(-1)\oplus\cdots\oplus \bfQ(-m)\right)^{\be_0=0}\cong \bfQ(0).
	\]
	Hence, we have that
	\[
	\left(\mathrm{Sym}^{2u}(V)(v)\right)^{\boldsymbol e_0=0}=\left(\mathrm{Sym}^{2u}(V)(v)\right)_{-2v,2u-2v}
	\]
 for any integers $u,v$ such that $u\geq 0$.
	Therefore, ${\fraku}_{-2(2n+1),-2n-2}^{\boldsymbol e_0=0}$ is contained in
	the $\mathrm{Sym}^{2n}(V)(2n+1)$-component of the $\GL_{2,\bfQ_p}$-module $\fraku$.
	
	Recall that there exists an isomorphism
	\[
	\mathrm{Sym}^k(V)\otimes_{\mathbf Q_p}\mathrm{Sym}^l(V)\cong \mathrm{Sym}^{k+l}(V)\bigoplus \mathrm{Sym}^{k+l-2}(V)(-1)\bigoplus\cdots \bigoplus \mathrm{Sym}^{k-l}(V)(-l)
	\]
	of $\GL_{2,\bfQ_p}$-modules for any $k\geq l\geq 0$. Therefore, any irreducible component appearing in
	$
	\mathrm{Sym}^k(V)(k+1+r)\otimes_{\mathbf Q_p}\mathrm{Sym}^l(V)(l+1+s)$ with $r,s\geq 0$ is of the form
	\[
	\Sym^u(V)(u+1+w),\quad w\geq 1.
	\]
	Therefore, for each non-negative integers $n_i,r_i$, there exists no $\mathrm{Sym}^{2n}(V)(2n+1)$-component
	in the $\GL_{2,\bfQ_p}$-module
	\[
	\mathrm{Sym}^{n_1}(V)(n_1+1+r_1)\bigotimes\cdots \bigotimes \mathrm{Sym}^{n_a}(V)(n_a+1+r_a)
	\]
	when $a$ is greater than one. Let $\Gamma^i\fraku$ be the central descending filtration of $\fraku$ defined by
	$\Gamma^1\fraku=\fraku,\ \Gamma^i\fraku=[\Gamma^{i-1}\fraku,\fraku],\ i\geq 2$. 
	Since the natural surjective homomorphism
	\[
	\mathrm H_1(\fraku)^{\otimes a}\twoheadrightarrow {\Gr}_\Gamma^a\fraku;\quad \overline{x_1}\otimes\cdots\otimes\overline{x_a}\mapsto \overline{[x_1,[x_2,[\cdots[x_{a-1},x_a]\cdots]}
	\]
	is $\GL_{2,\bfQ_p}$-equivariant, there exists no $\mathrm{Sym}^{2n}(V)(2n+1)$-component in ${\Gr}_\Gamma^a\fraku:=\Gamma^a\fraku/\Gamma^{1+1}\fraku$ if $a$ is greater than
	one.
	This implies that
	\[
	{\fraku}_{-2(2n+1),-2n-2}^{\boldsymbol e_0=0}\bigcap \Gamma^2\fraku=\{0\}.
	\]
	Therefore, we have an isomorphism
	\[
	{\fraku}_{-2(2n+1),-2n-2}^{\boldsymbol e_0=0}\xrightarrow{\sim}\mathrm H_1(\fraku)_{-2(2n+1),-2n-2}^{\boldsymbol e_0=0}\cong\left(\mathrm{Sym}^{2n}(V)(2n+1)\right)^{\boldsymbol e_0=0}
	\]
	by Proposition \ref{prop1-5} and this finishes the proof of the lemma.
\end{proof}
Let us define the bigraded Lie algebra $\fraku_{*,\bullet}^{\mathrm{g}}$
by
\begin{equation}
	\label{eq8}
	{\fraku}_{*,\bullet}^{\mathrm{g}}={\Gr}^{W^\infty}_*{\Gr}^W_\bullet\fraku\cong\bigoplus_{m,n<0}{\fraku}_{2m,n}^{\mathrm g}.
\end{equation}
\begin{definition}
	\label{dfn1-8}We define the Lie subalgebra
	$\fraku^{\mathrm g}_+$ of $\fraku^{\mathrm g}_{*,\bullet}$
	by
	\[
	{\fraku}^{\mathrm g}_+=\left\langle\boldsymbol e_0^{2n}\boldsymbol e_{2n+2}\ \middle |n\geq 1\right\rangle .
	\]
\end{definition}
\section{The canonical subalgebra of the depth-graded motivic Lie algebra}\label{depth}
In this section, we introduce the canonical Lie subalgebra
of the motivic depth-graded Lie algebra
and relate it to ${\fraku}^{\mathrm g}_+$.
\subsection{The depth-graded motivic Lie algebra}
For $*=\dR,\mathrm B$, let
\[
\omega_*\colon \MT(\mathbf Z)\to \mathrm{Vec}_{\mathbf Q}^{\mathrm{fin}}
\]
be the composition of the $*$-realization functor and the forgetful functor.
Let $\pi_1^{\mathrm{un}}(\MT(\mathbf Z),\omega_*)$
be the prounipotent radical of $\pi_1(\MT(\mathbf Z),\omega_*)$ and let $\frakt_*$ be the graded Lie algebra
defined by
\[
\frakt_*={\Gr}^M_*\mathrm{Lie}\left(\pi_1^{\mathrm{un}}(\MT(\mathbf Z),\omega_{\mathrm B})\right),
\]
where $M_*$ denote the weight filtration induced by that of mixed Tate motives over $\bfZ$.
Note that $\frakt_*$ is naturally isomorphic
to ${\Gr}^M_*\mathrm{Lie}\left(\pi_1^{\mathrm{un}}(\MT(\mathbf Z),\omega_\dR)\right)$
because the fiber functor $\omega_{\mathrm B}\circ\Gr^{M}$
is naturally isomorphic to $\omega_{\dR}$.
According to \cite[(2.4.1), (2.1.3)]{De-G05}
there exists a natural isomorphism
\[
\mathrm H_1(\frakt_*)\cong \bigoplus_{r>0}{\Hom}_{\mathbf Z}(K_{2r+1}(\mathbf Q),\mathbf Q)
\]
of graded $\mathbf Q$-vector spaces, where $K_{2r+1}(\mathbf Q)$ is Quillen's higher $K$-group. Here, the degree of ${\Hom}(K_{2r+1}(\mathbf Q),\mathbf Q)$ is
defined to be $-2(2r+1)$. 
Let us \emph{choose} an element $\boldsymbol{\sigma}_{2r+1}$ of $\frakt_{-2(2r+1)}$
such that its image under the homomorphism
\[
\frakt_{-2(2r+1)}\to {\Hom}(K_{2r+1}(\mathbf Q),\mathbf Q)\subset {\Hom} (K_{2r+1}(\mathbf Q),\mathbf R)\isom{\Hom}_{\mathbf R}(\mathbf R,\mathbf R)
\]
sends $\zeta(2m+1)$ to one. Here, the last isomorphism above is induced by the Beilinson
regulator map.
Then, the Lie algebra $\frakt_*$ is generated by $\{\boldsymbol{\sigma_{2m+1}}\}_{m\geq 1}$
freely (\cite[Proposition 2.3]{De-G05}).
Note that such a choice is not unique \emph{at all}. See \cite{Brown15}, \cite{Brown17} for interesting trials
to find normalized lifts.

Let $\Pi$ be the de Rham realization of the motivic fundamental group of $\mathbf P^1_{\mathbf Q}\setminus\{0,1,\infty\}$
with the base point $\overrightarrow{01}$.
Then we have natural isomorphisms
\begin{equation}
	\label{eq9}
	\mathcal O(\Pi)=\bigoplus_{n\geq 0}\mathrm H^1_{\dR}(\mathbf P^1_{\mathbf Q}\setminus\{0,1,\infty\})^{\otimes n}=\mathbf Q\left\langle \frac{dt}{t},\frac{dt}{1-t}\right\rangle
\end{equation}
by \cite[3.13]{De-G05}.
Let $\mathbf Q\langle\!\langle x_0,x_1\rangle\!\rangle$
denote the dual of $\mathcal O(\Pi)$ such that each word in $x_0$ and $x_1$ is
dual to the corresponding word in $\frac{dt}{t}$ and $\frac{dt}{1-t}$.
Then, the graded Lie algebra
$\Gr^M_\bullet\mathrm{Lie}(\Pi)$
is naturally identified with the free Lie algebra
generated by $x_0$ and $x_1$:
\[
{\Gr}^M_*\mathrm{Lie}(\Pi)=\mathrm{Lie}(x_0,x_1)\subset \mathbf Q\langle x_0,x_1\rangle.
\]
We define $D^1\mathrm{Lie}(x_0,x_1)$ to be the kernel of the Lie homomorphism
\[
\mathrm{Lie}(x_0,x_1)\to \mathrm{Lie}(x_0);\quad x_0\mapsto x_0,\ x_1\mapsto 0.
\]
For $n\geq 2$,  $D^n\mathrm{Lie}(x_0,x_1)$ is defined by the equation
\[
D^n\mathrm{Lie}(x_0,x_1)=[D^1\mathrm{Lie}(x_0,x_1),D^{n-1}\mathrm{Lie}(x_0,x_1)].
\]
Note that the action of $\frakt_*$ on $\Lie(x_0,x_1)$ is factored as follows (\cite[(2.4), (2.5)]{Brown17}):
\[
\frakt_*\xrightarrow{i}\mathrm{Lie}(x_0,x_1)\xrightarrow{a}\mathrm{Der}(\mathrm{Lie}(x_0,x_1)).
\]	
Here, the homomorphism $a$ is given by the equations
\[
a(X)(e_0)=[e_0,X],\quad a(X)(e_1)=0.
\]
\begin{definition}
	\label{dfn2-1}\begin{itemize}
		\item[(1)]For a positive integer $d$,
		we define $D^d\frakt_*$ by
		\[
		D^d\frakt_*=i^{-1}(D^d(\Lie(x_0,x_1)))
		\]
		We call this descending filtration $D^\bullet\frakt_*$ the \emph{depth filtration}
		of $\frakt_*$.
		\item[(2)]The \emph{motivic depth-graded Lie algebra $\mathfrak d$} is defined by
		\[
		\mathfrak d=\bigoplus_{d>0}{\Gr}_D^d\frakt_*.
		\]
	\end{itemize}
\end{definition}
The Lie algebra $\mathfrak d$ has a natural structure of
a bi-graded Lie algebra.
\begin{example}
	\label{ex1}
	According to \cite[(2.6)]{Brown17}  we have a congruence
	\[
	i(\boldsymbol{\sigma}_{2m+1})\equiv \mathrm{ad}^{2m}(x_0)x_1\quad (\mathrm{mod}\ D^2\Lie(x_0,x_1)).
	\]
	Therefore, the image $\overline{\boldsymbol{\sigma}}_{2m+1}$ of $\boldsymbol{\sigma}_{2m+1}$
	in $\mathfrak d$ is non-trivial and does not depend on the choice of $\boldsymbol{\sigma}_{2m+1}$.
	In this paper, we call them \emph{canonical generators} of $\mathfrak d$ (cf.\ \cite[Subsection 1.4.1]{Brown15}).
\end{example}
\begin{definition}[{\cite[Definition 6.8]{Brown15}}]
	\label{cansubalg}Define a subalgebra $\mathfrak d^{\mathrm{odd}}$ of $\mathfrak d$
	to be the subalgebra generated by canonical generators:
	\[
	\mathfrak d^{\mathrm{odd}}=\left\langle \overline{\boldsymbol{\sigma}}_{2m+1}\ \middle |\ m\geq 1\right\rangle.
	\]
\end{definition}

\subsection{The depth-weight compatibility hypothesis}\label{Hyp}
The functor $\bv^*\colon \ME_p(\scrM_{1,1})\to {\MT}_p(\mathbf Z)$
and the equivalence in Proposition \ref{prop1-2}
induces a natural injective homomorphism
\begin{equation}
	\label{eq10}
	{\bv}_*\colon \frakt_*\otimes_{\mathbf Q}\mathbf Q_p\hookrightarrow {\Gr}^{W^\infty}_*\fraku
\end{equation}
of the graded Lie algebras over $\mathbf Q_p$.
Note that, by definition, the weight filtration $M_*\mathrm{Lie}({\MT}(\mathbf Z),\omega_{\rmB})$
and the local weight filtration $W^\infty_*\mathrm{Lie}({\ME}_p(\scrM_{1,1}),\omega_{\rmB})$
are strictly compatible under the homomorphism $\bv_*$.
We define the \emph{global weight filtration $W_\bullet(\frakt_*\otimes_{\mathbf Q}\mathbf Q_p)$} by the equation
\begin{equation}
	\label{eq11}
	W_\bullet\left(\frakt_*\otimes_{\mathbf Q}\mathbf Q_p\right)={\bv}_*^{-1}\left(W_\bullet{\Gr}^{W^\infty}_*\fraku\right).
\end{equation}
Our main hypothesis is that there exists a compatibility of the depth filtration
and the global filtration:
\begin{hypothesis}
	\label{hyp2-2}Let $n$ and $m$ be negative integers. Then, we have the equality
	\[
	W_n(\frakt_{2m}\otimes_{\mathbf Q}\mathbf Q_p)=D^{m-n}\frakt_{2m}\otimes_{\mathbf Q}\mathbf Q_p.
	\]
	In particular, $W_\bullet(\frakt_*\otimes_{\mathbf Q}\mathbf Q_p)$ is defined over $\mathbf Q$
	and does not depend on the choice of $p$.
\end{hypothesis}
\begin{remark}\label{remhyp}
	Let $\mathfrak p$ denote the Lie algebra of the unipotent
	fundamental group of the punctured Tate elliptic curve with the Weierstrass
	tangential base point (\cite[Section 22]{HM15}) and let
	\[
	\mathfrak u\to \mathrm{Der}(\mathfrak p\widehat{\otimes}_{\mathbf Q}\mathbf Q_p)
	\]
	be the monodromy representation.
	If the above morphism is injective, then the hypothesis \ref{hyp2-2}
	is true by \cite[Corollary 28.10]{HM15}.
	
	Note that the monodromy representation above induced by the family $\scrM_{1,2}\to \scrM_{1,1}$,
	which is an elliptic analogue of the family $\scrM_{0,4}=\bfP^1\setminus\{0,1,\infty\}\to \scrM_{0,3}=\spec(\bfQ)$.
	The monodromy representation arising from the latter family is nothing but the outer action of $\scrG_{\bfQ}$ on $\pi_1^{\et}(\bfP^1_{\overline{\bfQ}}\setminus\{0,1,\infty\})$.
	Therefore, the injectivity of the monodromy representation can be regarded as an elliptic analogue of Brown's faithfulness theorem (\cite{Brown12}). A motivic version of this problem has already been
	raised in \cite[Question 26.2]{HM15}.
\end{remark}
Under the hypothesis above, the Lie homomorphism $\bv_*$ induces an inclusion
\begin{equation}
	\label{eq12}
	{\bv}_*\colon \mathfrak d\otimes_{\mathbf Q}\mathbf Q_p\hookrightarrow{\Gr}^W_\bullet{\Gr}^{W^\infty}_*\fraku
\end{equation}
of Lie algebras.
We put
\begin{equation}
	\label{eq16}
	\breve{\boldsymbol e}_{2m+2}:=\frac{1}{(2m)!}\boldsymbol e_0^{2m}\boldsymbol e_{2m+2}.
\end{equation}
The following proposition can be thought as a kind of ``Eisenstein congruence'':
\begin{proposition}
	\label{keyprop}
	Under Hypothesis $\ref{hyp2-2}$, we have that
	\[
	{\bv}_*(\overline{\boldsymbol{\sigma}}_{2m+1})=\breve{\boldsymbol e}_{2m+2}.
	\]
	In particular, $\bv_*$ induces a natural isomorphism
	\[
	\mathfrak d^{\mathrm{odd}}\otimes_{\mathbf Q}\mathbf Q_p\xrightarrow{\sim}\mathfrak u_+^{\mathrm g}
	\]
	of Lie algebras.
\end{proposition}
For a proof of the proposition above, we prepare a lemma.
\begin{lemma}
	\label{monodoromyinvariance}
	The image of $\mathfrak d\otimes_{\mathbf Q}\mathbf Q_p$ under $\bv_*$ is annihilated
	by the adjoint action of $\boldsymbol e_0$.
\end{lemma}
\begin{proof}
	Let
	\[
	\rho_{\mathrm{univ}}\colon \pi_1^{\et}(\scrM_{1,1},v)=\widehat{\mathrm{SL}_2(\mathbf Z)}\rtimes \mathscr G_{\mathbf Q}\to \pi_1({\ME}_p(\scrM_{1,1}),v)(\mathbf Q_p)
	\]
	be the universal representation, which is induced by the natural fully faithful functor
	\[
	{\ME}_p(\scrM_{1,1})\hookrightarrow\mathrm{Sm}_{\mathbf Q_p}(\scrM_{1,1}).
	\]
	Let us define an element $\varepsilon$ of $\pi_1({\ME}_p(\scrM_{1,1}),v)(\mathbf Q_p)$ to be the image
	of $\begin{pmatrix}
		1&1\\
		0&1
	\end{pmatrix}\in \mathrm{SL}_2(\mathbf Z)$ under $\rho_{\mathrm{univ}}$.
	Then, since the image of $\varepsilon$ under the natural surjection
	$ \pi_1({\ME}_p(\scrM_{1,1}),v)\to \GL_{2,\bfQ_p}$ coincides with
	$\exp(\boldsymbol e_0)$, we have
	\begin{equation}
		\label{eq14}
		\mathrm{spl}(\exp(\boldsymbol e_0))=\varepsilon u
	\end{equation}
	for some $u\in U(\mathbf Q_p)$.
	
	Since the pull-back by $v$ induces a continuous injective homomorphism
	\begin{equation}
		\label{eq13}
		\mathrm{Gal}\left(\overline{\mathbf Q}\{\!\{q\}\!\}/\mathbf Q(\!(q)\!)\right)=\widehat{\mathbf Z}(1)\rtimes \mathscr G_{\mathbf Q}\hookrightarrow \pi_1^{\et}(\scrM_{1,1},v)
	\end{equation}
	and a generator of $\widehat{\mathbf Z}(1)$ is sent to $\begin{pmatrix}
		1&1\\0&1
	\end{pmatrix}$ under the map
	\[
	\widehat{\mathbf Z}(1)\rtimes \mathscr G_{\mathbf Q}\hookrightarrow\pi_1^{\et}(\scrM_{1,1},v)\to \GL_2(\widehat{\bfZ}),
	\]
	any element of $\rho_{\mathrm{univ}}(\mathscr G_{\mathbf Q(\mu_{p^\infty})})$
	commutes with $\varepsilon$.
	Therefore, the image of $\frakt_*\otimes\mathbf Q_p$ under $\bv_*$ stabilized by the adjoint action of
	$\varepsilon$ because the universal representation
	\[
	\mathscr G_{\mathbf Q(\mu_{p^\infty})}\to \pi_1^{\mathrm{un}}({\MT}_p(\mathbf Z),v)(\mathbf Q_p)
	\]
	has a Zariski dense image by \cite[Theorem 8.7]{HM03} and \cite[Theorem 1.1]{Brown12}.
	Here, $\pi_1^{\mathrm{un}}({\MT}_p(\mathbf Z),v)$ is the prounipotent
	radical of $\pi_1({\MT}_p(\mathbf Z),v)$.
	Therefore, $\log(\epsilon) \in \Lie(\pi_1(\ME(\scrM_{1,1}),v))$ commutes with any element of $\frakt_*\otimes\mathbf Q_p$.
	By the Campbell-Hausdorff theorem (\cite[Theorem 7.4]{SerreLie}), we have
	\[
	\mathrm{spl}'(\be_0)=\sum_{n=1}^\infty z_n(\log(\varepsilon),\log(u))
	\]
	for some homogeneous Lie polynomials $z_n(X,Y)$ of degree $n$.
	Note that the adjoint action of $\log(u)$ on
	$\mathfrak u$ decreases their global weights because any global weight of $\fraku$ is negative.
	Therefore,	the homomorphism
	\[
	\mathrm{ad}(\mathrm{spl}'(\be_0))\colon \bv_*(\frakt_*\otimes_{\mathbf Q}\mathbf Q_p)\to  {\Gr}^{W^\infty}_*\fraku
	\] 
	also decreases their global weights.
	
	For $x\in \bv_*(\frakt_{-2m}\otimes_{\mathbf Q}\mathbf Q_p)$,
	we expand $x$ as
	\[
	x=\sum_{r<0}x^{(r)},
	\]
	where $x^{(r)}\in \mathfrak u_{-2m,r}$.
	Let $r_0$ be the maximal integer such that $x^{(r_0)}$ is non-zero
	and we call $x^{(r_0)}$ the leading term of $x$.
	By the argument of the previous paragraph, the weight of $[\mathrm{spl}'(\be_0),x]$
	is less than $r_0$.	On the other hand, as the action of $\mathrm{Lie}(\GL_{2,\bfQ_p}))$ on $\fraku$ preserves the global weight splitting
	\[
	\Gr^W_r\fraku=\prod_{m<0}\mathfrak u_{-2m,r}\subset \fraku=\prod_{m,r<0}\mathfrak u_{-2m,r}
	\]
	by the definition of the splitting, $\mathrm{ad}(\mathrm{spl}'(\be_0))(x^{(r)})$ is contained in $\prod_{m<0}\mathfrak u_{-2m,r}$ for any $r$.
	This implies that the equation
	\[
	[\mathrm{spl}'(\be_0),x^{(r_0)}]=0
	\]
	holds in $\fraku$.
	By Hypothesis \ref{hyp2-2}, the Lie algebra $\bv_*(\mathfrak d\otimes_{\mathbf Q}\mathbf Q_p)$
	is the Lie algebra generated by leading terms of
	elements of $\bv_*(\frakt_*\otimes_{\mathbf Q}\mathbf Q_p)$,
	we have the conclusion of the lemma.
\end{proof}
\begin{proof}[Proof of Proposition $\ref{keyprop}$]Recall that the depth of $\overline{\boldsymbol{\sigma}}_{2m+1}$ is one (Example \ref{ex1}) and its local weight is equal to $-2(2m+1)$. Moreover,
	by Lemma \ref{monodoromyinvariance}, $\bv_*(\overline{\boldsymbol{\sigma}}_{2m+1})$ is annihilated
	by $\boldsymbol e_0$. Therefore, by Lemma \ref{lem1-7}, we have the equation
	\[
	{\bv}_*(\overline{\boldsymbol{\sigma}}_{2m+1})=\alpha \boldsymbol e_0^{2m}\boldsymbol e_{2m+2}
	\]
	for some $\alpha\in \mathbf Q_p^\times$.
	To determine the constant $\alpha$, it is sufficient to
	compute the image of $\bv_*(\overline{\boldsymbol{\sigma}}_{2m+1})$ under the
	monodromy representation on the graded Lie algebra associated with the punctured
	Tate elliptic curve.
	This was already computed by Hain-Matsumoto (\cite[Theorem 29.4]{HM15})
	and we conclude that $\alpha$ is equal to $1/(2m)!$.
	This finishes the proof of the proposition.
\end{proof}
\section{Lower bounds on the space of cup products of Eisenstein classes}\label{injectivity}
Let $m,i,j$ be non-negative integers such that $m=i+j$.
Then, the natural morphism
\[
\calV^{2i}\otimes_{\mathbf Q_p}\calV^{2j}\to \calV^{2m}
\]
in ${\ME}_p(\scrM_{1,1})$ induces a homomorphism
\begin{multline}
	\label{cupprod}
	\cup\colon {\Ext}^1_{{\ME}_p(\scrM_{1,1})}(\mathbf Q_p,\calV^{2i}(2i+1))\otimes_{\mathbf Q_p}{\Ext}_{{\ME}_p(\scrM_{1,1})}^1(\mathbf Q_p,\calV^{2j}(2j+1))\\
	\longrightarrow {\Ext}_{{\ME}_p(\scrM_{1,1})}^2(\mathbf Q_p,\calV^{2m}(2m+2)).
\end{multline}
We call the map $\cup$ a \emph{cup product} simply.
In this section, assuming Hypothesis \ref{hyp2-2}, we give a lower bound for
the dimension of the space generated by 
$\{\varphi_{2i+2}\cup\varphi_{2j+2}\}_{i+j=m}$ in the second extension group above.
According to the proof of \cite[Proposition 13.1]{HM15}, we have a natural isomorphism
\[
 {\Ext}^1_{{\ME}_p(\scrM_{1,1})}(\mathbf Q_p,\calV^{2i}(2i+1))\cong \rmE_{2i}:=\rmH^1_{\et}(\scrM_{1,1},\calV^{2i}(2i+1)).
\]
Therefore, the cup product above can be regarded as a homomorphism
\[
\rmE_{2i}\otimes_{\bfQ_p}\rmE_{2j}\to {\Ext}_{{\ME}_p(\scrM_{1,1})}^2(\mathbf Q_p,\calV^{2m}(2m+2))
\]
naturally.
By abuse of notation, we also write $\cup$ the twisted homomorphism
\[
\rmE_{2i}(-2i-1)\otimes_{\bfQ_p}\rmE_{2j}(-2j-1)\to {\Ext}_{{\ME}_p(\scrM_{1,1})}^2(\mathbf Q_p,\calV^{2m}(2m+2))(-2m-2)
\]
by the $(-2m-2)$-times of the standard character of $\bfG_m$.
\subsection{Relations among canonical generators}
\begin{definition}[cf.\ {\cite[Definition 7.1]{Brown15}}]
	\label{dfn3-1}A \emph{restricted even period polynomial of degree $2m$}
	is a two-variable homogeneous polynomial $f(x,y)$ with coefficients in $\mathbf C$ of degree $2m$
	satisfying the following four equations:
	\begin{equation}
		f(x,0)=0,
	\end{equation}
	\begin{equation}
		f(\pm x,\pm y)=f(x,y),
	\end{equation}
	\begin{equation}
		f(x,y)+f(y,x)=0,
	\end{equation}
	\begin{equation}
		f(x,y)+f(x-y,x)+f(-y,x-y)=0.
	\end{equation}
	Let $S_{2m+2}$ denote the set of the restricted period
	polynomials of degree $2m$ with \emph{rational coefficients}.
\end{definition}
\begin{theorem}
	\label{thm3-2}Let $m$ be a positive integer and let
	$(a_{i,j})_{i+j=m,\ i,j\geq 1,\ i<j}$ be a set of rational numbers.
	Then, the equation
	\[
	\sum_{i<j,\ i+j=m}a_{i,j}\left[\overline{\boldsymbol \sigma}_{2i+1},\overline{\boldsymbol \sigma}_{2j+1}\right]=0
	\]
	holds in $\mathfrak d$ if and only if
	the polynomial
	$
	\sum_{i<j,\ i+j=m}a_{i,j}(x^{2i}y^{2j}-y^{2i}x^{2j})
	$ is an element of $S_{2m+2}$.
\end{theorem}
\begin{proof}
In \cite[Definition 6.7]{Brown17}, Brown defined a graded Lie algebra $\bigoplus_{n\geq 1}\mathbb D_n$ and an injective Lie homomorphism
\[
\overline{\rho}\colon \frakd\intomap \bigoplus_{n\geq 1}\mathbb D_n
\]
which is isomorphism in degree one.
Therefore, the left-side hand in the equation above is equal to zero if and only if the equation
\[
\sum_{i<j,\ i+j=m}a_{i,j}\left[\overline{\rho}(\overline{\boldsymbol \sigma}_{2i+1}),\overline{\rho}(\overline{\boldsymbol \sigma}_{2j+1})\right]=0
\]
holds.
Then, the conclusion of the theorem is a direct consequence of \cite[(7.7)]{Brown17}.
\end{proof}
\begin{corollary}
	\label{cor3-3}Let us use the same notation as the theorem above.
	Moreover, we suppose that Hypothesis \ref{hyp2-2} is true.
	Then, the equation
	\[
	\sum_{i<j,\ i+j=m}a_{i,j}\left[\breve{\boldsymbol e}_{2i+2},\breve{\boldsymbol e}_{2i+2}\right]=0
	\]
	holds in $\mathfrak u_+^{\mathrm g}$ \emph{if and only if}
	the polynomial
	$
	\sum_{i<j,\ i+j=m}a_{i,j}(x^{2i}y^{2j}-y^{2i}x^{2j})
	$ is an element of $S_{2m+2}$.
\end{corollary}
\begin{proof}This is a direct consequence of Theorem
	\ref{thm3-2} and Proposition \ref{keyprop}.
\end{proof}
\begin{corollary}
	\label{cor3-4}Suppose that Hypothesis \ref{hyp2-2} is true. Then, under the identification
	\[
	\bigwedge^2\mathrm H_1(\mathfrak u_+^{\mathrm g})\xrightarrow{\sim}\bigoplus_{1\leq i<j}\mathbf Q(x^{2i}y^{2j}-y^{2i}x^{2j});\quad \breve{\boldsymbol e}_{2i+2}\wedge\breve{\boldsymbol e}_{2j+2}\mapsto x^{2i}y^{2j}-y^{2i}x^{2j},
	\]
	the image of the dual of the cup product $($cf.\ \cite[Section 18]{HM15}$)$
	\[
	\mathrm H_2(\mathfrak u_+^{\mathrm g})\to \bigwedge^2\mathrm H_1(\mathfrak u_+^{\mathrm g})
	\]
	coincides with $\bigoplus_{m\geq 6}S_{2m}\otimes_{\mathbf Q}\mathbf Q_p$.
\end{corollary}
Note that we have a natural isomorphism
\[
\rmH^1(\fraku_+^{\mathrm g})=\bigoplus_{i\geq 2}\bfQ_p\breve{\be_{2i}}\cong \bigoplus_{i\geq 2}\rmE_{2i}^\vee(2i+1).
\]
(see Proposition \ref{prop1-5}). Therefore, by taking the dual of the homomorphism in Corollary \ref{cor3-4}, we have a natural surjective homomorphism
\begin{equation}
	\label{eq26}
\bigoplus_{i+j=2m}\rmE_{2i}(-2i-1)\otimes_{\bfQ_p}\rmE_{2j}(-2j-1)\ontomap S_{2m+2,\bfQ_p}^\vee
\end{equation}
where $S_{2m+2,\bfQ_p}:=S_{2m+2}\otimes_{\bfQ}\bfQ_p$ and $S_{2m+2,\bfQ_p}^\vee$ denotes the $\bfQ_p$-dual of $S_{2m+2,\bfQ_p}$.
\subsection{Non-vanishing of cup products of Eisenstein classes}
Let us show non-vanishing of cup products of $p$-adic \'etale Eisenstein classes.

According to \cite[Theorem 5.3]{HM04}, we have a natural isomorphism
\begin{equation}
	\label{eq17}
	\mathrm H_i({\fraku}_{*,\bullet})=\bigoplus_{n\geq 0,r\in \mathbf Z}{\Ext}^i_{{\ME}_p(\scrM_{1,1})}(\mathbf Q_p,\mathcal V^{n}(r))^\vee\otimes_{\mathbf Q_p}V^n(r).
\end{equation}
Then, for each $n\geq 0$ and $r\in \mathbf Z$, we define a homomorphism
\[
\alpha_{n,r}\colon {\Ext}^2_{{\ME}_p(\scrM_{1,1})}(\mathbf Q_p,\mathcal V^{n}(r))(-r)\to \mathrm H_2({\fraku}_{+}^{\mathrm g})^\vee
\]
to be the composition
\begin{multline} {\Ext}^2_{{\ME}_p(\scrM_{1,1})}(\mathbf Q_p,\mathcal V_{n}(r))(-r)
	\hookrightarrow{\Ext}^2_{{\ME}_p(\scrM_{1,1})}(\mathbf Q_p,\mathcal V_{n}(r))\otimes_{\mathbf Q_p}\left(V_n(r)\right)^\vee\\
	\hookrightarrow \mathrm H_2({\fraku}_{*,\bullet})^\vee\to \mathrm H_2({\fraku}_+^{\mathrm g})^\vee.
\end{multline}
Here, the first inclusion is the induced homomorphism by the inclusion
\[
\bfQ_p(-r)\intomap  \left(V_n(r)\right)^\vee.
\]
\begin{theorem}
	\label{thm3-5} We have the following commutative diagram$:$
{\small	\[
	\xymatrix{
		\displaystyle{\bigoplus_{i+j=m}}\mathrm{E}_{2i+2}(-2i-1)\otimes_{\mathbf Q_p}\mathrm{E}_{2j+2}(-2j-1)\ar[rr]^{\cup}\ar@{->>}[d]& & {\Ext}^2_{{\ME}_p(\scrM_{1,1})}(\mathbf Q_p,\mathcal V^{2m}(2m+2))(-2m-2)\ar[d]^{\alpha_{2m,2m+2}}\\
		S_{2m+2,\bfQ_p}^\vee\ar@{^(->}[rr]& &\mathrm H_2({\fraku}_+^{\mathrm g})^\vee.
	}
	\]}
	Here, the left vertical homomorphism is the same as the one in (\ref{eq26}).
\end{theorem}
\begin{proof}
	The inclusion $\mathfrak u_+^{\mathrm g}\subset {\fraku}_{*,\bullet}$ induces the
	homomorphisms of homology groups. Then, we have the following commutative diagram:
	\[
	\xymatrix{
		\mathrm H_2({\fraku}_+^{\mathrm g})\ar[d]\ar[rr]& & \bigwedge^2\mathrm H_1({\fraku}_+^{\mathrm g})\ar@{^(->}[d]\\
		\mathrm H_2({\fraku}_{*,\bullet})\ar[rr]& &\bigwedge^2\mathrm H_1({\fraku}_{*,\bullet}).
	}
	\]
	Note that the images of the horizontal homomorphisms are annihilated by
	$\mathrm{ad}(\boldsymbol e_0)$ because ${\fraku}_+^{\mathrm g}$ is annihilated by $\mathrm{ad}(\boldsymbol e_0)$.
	Thus, for any positive integers $i,j $ such that $m=i+j$, we have the commutative diagram
	{\small\[
	\xymatrix{
		\mathrm H_2({\fraku}_+^{\mathrm g})\ar[d]_{{\alpha_{2m,2m+2}}^\vee}\ar[rr]& & \mathrm{E}^\vee_{2i+2}(2i+1)\otimes_{\mathbf Q_p}\mathrm{E}^\vee_{2j+2}(2j+1)\ar[d]^{\cong}\\
		{\Ext}^2_{{\ME}_p(\scrM_{1,1})}(\mathbf Q_p,\mathcal V^{2m}(2m+2))^\vee(2m+2)\ar[rr]& & \mathrm{E}^\vee_{2i+2}\otimes_{\mathbf Q_p}\mathrm{E}^\vee_{2j+2}\otimes_{\mathbf Q_p} V_{2i}^{\boldsymbol e_0=0}\otimes V_{2j}^{\boldsymbol e_0=0}(2m+2)
	}
	\]}by the natural isomorphism (\ref{eq17}).
	Therefore, by taking the dual of the diagram above and by taking their sum, we have the following commutative diagram:
	{\small\[
	\xymatrix{
		{\displaystyle\bigoplus_{i+j=m}}\mathrm E_{2i+2}(-2i-1)\otimes_{\mathbf Q_p}\mathrm E_{2j+2}(-2j-1)\ar[rr]^{\cup}\ar[d]^{=}& &{\Ext}^2_{{\ME}_p(\scrM_{1,1})}(\mathbf Q_p,\mathcal V^{2m}(2m+2))(-2m-2)\ar[d]^{\alpha_{2m,2m+2}}\\
			{\displaystyle\bigoplus_{i+j=m}}\mathrm E_{2i+2}(-2i-1)\otimes_{\mathbf Q_p}\mathrm E_{2j+2}(-2j-1)\ar[rr]^{\cup}& &\mathrm H_2({\fraku}_+^{\mathrm g})^\vee.
	}
	\]
}Thus, we have the conclusion of the theorem by the diagram above and
	Corollary \ref{cor3-4}.
\end{proof}
\begin{corollary}
	\label{cor3-6}Under Hypothesis \ref{hyp2-2}, we have an inequality that
	\[
	\dim_{\mathbf Q_p}\left(\sum_{i+j=m}\mathbf Q_p\varphi_{2i+2}\cup\varphi_{2j+2}\right)\geq \dim_{\mathbf Q}S_{2m+2}.
	\]
\end{corollary}
\begin{proof}
	This is a direct consequence of Theorem \ref{thm3-5}.
\end{proof}
By Drinfeld-Manin's theorem, we have a decomposition
\[
\mathrm H^1_{\et}\left({\scrM_{1,1}}/\overline{\mathbf Q},\mathcal V^{2m}\right)=\mathrm E_{2m+2}(-2m-1)\bigoplus\mathrm H^1_{\mathrm{cusp}}\left({\scrM_{1,1}}/\overline{\mathbf Q},\mathcal V^{2m}(2m+2)\right)
\]
of $\bfQ_p$-vector spaces stable under the action of $\scrG_{\bfQ}$,
where $\mathrm H^1_{\mathrm{cusp}}\left({\scrM_{1,1}}/\overline{\mathbf Q},\mathcal V^{2m}(2m+2)\right)$
is the subrepresentation of $\mathrm H^1_{\et}\left({\scrM_{1,1}}/\overline{\mathbf Q},\mathcal V^{2m}(2m+2)\right)$
corresponding to cuspforms.
According to \cite[Theorem 13.3]{HM15},
the natural homomorphism
\begin{equation}
	{\Ext}^2_{{\ME}_p(\scrM_{1,1})}(\mathbf Q_p,\mathcal V^{2m}(2m+2))\to \mathrm H^1\left(\mathbf Q,\mathrm H^1\left({\scrM_{1,1}}/\overline{\mathbf Q},\mathcal V^{2m}(2m+2)\right)\right)
\end{equation}
is injective and its image is contained in the Bloch-Kato finite part
\[
\mathrm H^1_{\mathrm f}\left(\mathbf Q,\mathrm H^1\left({\scrM_{1,1}}/\overline{\mathbf Q},\mathcal V^{2m}(2m+2)\right)\right)=\mathrm H^1_{\mathrm f}\left(\mathbf Q,\mathrm H^1_{\mathrm{cusp}}\left({\scrM_{1,1}}/\overline{\mathbf Q},\mathcal V^{2m}(2m+2)\right)\right).
\]
Therefore, the $\bfQ_p$-subspace of ${\Ext}^2_{{\ME}_p(\scrM_{1,1})}(\mathbf Q_p,\mathcal V^{2m}(2m+2))$ spanned by elements $\{\varphi_{2i+2}\cup\varphi_{2j+2}|\  i+j=m\}$ can be regarded as a subspace of $\mathrm H^1_{\mathrm f}\left(\mathbf Q,\mathrm H^1_{\mathrm{cusp}}\left({\scrM_{1,1}}/\overline{\mathbf Q},\mathcal V^{2m}(2m+2)\right)\right)$ naturally.
\section{The Bloch-Kato conjecture for Galois representation associated with full-level modular forms}
In this section, under Hypothesis \ref{hyp2-2}, we give a proof of the Bloch-Kato conjecture
 for full-level Hecke eigen cuspforms.
\subsection{Kato Euler system}
Here, we recall the construction of a Kato Euler system briefly.
For simplification, we restricts ourselves to the full-level case.

Let $N$ be a positive integer which is divided by $p$.
Let $\calK$ be a finite extension of $\mathbf Q_p$ and let $T$ be a free $\mathcal O_{\calK}$-module equipped with a continuous action of $\mathscr G_{\mathbf Q}$
unramified outside $N$. Let us define the set $\Xi(N)$ of positive integers by
\[
\Xi(N):=\{n\in \mathbf Z_{\geq 1}\ |\ \mathrm{prime}(n)\cap\mathrm{prime}(N)=\{p\}\},
\]
where $\mathrm{prime}(n)$ denotes the set of prime divisors of $n$.
For any prime number $\ell$ which does not divide $N$,
let $P_\ell(t)$ denote the polynomial $\det(1-\sigma_\ell t\ |\ T)$,
where $\sigma_\ell$ is the arithmetic Frobenius at $\ell$.
Then, by an \emph{Euler system for $(T,N)$},
we mean a system of elements
\[
z_n\in \mathrm H^1\left(\mathbf Z\left[\mu_n,\frac{1}{p}\right],T\right),\quad n\in \Xi(N)
\]
satisfying the following norm relation:
Let $n,n'$ be elements of $\Xi(N)$ such that $n|n'$.
Then, we have
\[
\mathrm{Cor}_{\mathbf Q(\mu_{n'})/\mathbf Q(\mu_n)}(z_{n'})=\prod_{\ell|n',\ \ell\nmid n}P_\ell(\ell^{-1}\sigma_\ell^{-1})z_n
\]
(\cite[13.1]{Kat04}).
For a given Euler system $\{z_n\}_{n\in\Xi(N)}$, we
define $z_1$ by
\[
z_1=\mathrm{Cor}_{\mathbf Q(\mu_{p})/\mathbf Q}(z_{p}).
\]
By using a non-torsion Euler system, we can give
an ``upper bound'' on the second Iwasawa cohomology (\cite[Theorem 13.4]{Kat04}).
In this paper, we only use the following finiteness theorem:
\begin{theorem}[{\cite[Theorem 8.1]{KK99}}]Let $T$ be as above.
	We assume that the following two conditions are satisfied:
	\begin{itemize}
		\item The $\mathscr G_{\mathbf Q}$-module $T\otimes_{\mathbf Z}\mathbf Q$ is irreducible
		and pure.
		\item There exists an element $\sigma\in\mathscr G_{\mathbf Q^{\mathrm{ab}}}$
		such that $T/(1-\sigma)T$ is a rank one $\calO_{\calK}$-module.
	\end{itemize}
	Then, if the element $z_1$
	is \emph{not} a torsion element of $\mathrm H^1(\mathbf Z[1/p],T)$,
	then the second Galois cohomology group $\mathrm H^2(\mathbf Z[1/N],T)$
	is finite.
	\label{thm4-1}
\end{theorem}
\begin{remark}
In the paper \cite{KK99}, Kato defined an Euler system as a system of the groups
$\mathrm H^1\left(\mathbf Z\left[\mu_n,1/N\right],T\right)$ satisfying the same norm relations as above.
However, those elements in a system are automatically contained in
$ \mathrm H^1\left(\mathbf Z\left[\mu_n,1/p\right],T\right)$.
See \cite[Lemma 8.5]{Kat04}.
\end{remark}
Let us recall Kato's construction of an Euler system for the case
\[
T:=\mathrm H^1_{\mathrm{et}}\left({\scrM_{1,1}}/\overline{\mathbf Q},\calV^{2m}_{\mathbf Z_p}(2m+2)\right),
\]
where $ \calV^{2m}_{\mathbf Z_p}$ is defined by
\[
\calV^{2m}_{\mathbf Z_p}=\mathrm{Sym}^{2m}\rmR^1\pi_*\mathbf Z_p.
\]
See \cite[Section 1, Section 2, Section 8]{Kat04}
for the construction of Euler systems for general level.
For a positive integer $n$, let $Y(n)$ denote
the moduli stack over $\bfQ$ of elliptic curves with $\Gamma(n)$-level structure (\cite[(3.1)]{KM84}).
It is well known that $Y(n)$ is an actual scheme if $n$ is greater than two.
We take an auxiliary positive integers $c,d$ which are coprime to $6p$.
For $n\in \Xi(6pcd)$,
let $\kappa({_c}g_{1/np^\infty,0})$ and $\kappa({_d}g_{0,1/np^\infty})$
be elements of $\varprojlim_{r}\mathrm{H}^1_{\mathrm{et}}(Y(np^r),\mathbf Z/p^r\mathbf Z(1))$
which are defined to be the norm limit of Siegel units
${_c}g_{1/np^r,0},\ {_d}g_{0,1/np^r}\in \mathcal O\left(Y(np^r)\right)^\times$
(\cite[1.2]{Kat04}).
Let
$i,j$ be positive integers such that $m=i+j$.
Then, we define an element
\[
_{c,d}z_n(i,j)\in \mathrm H^1(\mathbf Q(\mu_n),T)
\]
to be the image of $\kappa({_c}g_{1/np^\infty,0})\cup\kappa({_d}g_{0,1/np^\infty})$
under the following composition of homomorphisms (\cite[(8.4.3)]{Kat04}):
\begin{multline*}
	\varprojlim_r\mathrm H^2_{\et}(Y(np^r),\mathbf Z/p^r\mathbf Z(2))\xrightarrow{\otimes e_{1,np^r}^{2i}e_{1,np^r}^{2j}}\varprojlim_r\mathrm H^2_{\et}(Y(np^r),\calV^{2m}_{\mathbf Z/p^r\mathbf Z_p}(2m+2))\\
	\longrightarrow \mathrm H^1(\mathbf Q,\mathrm H^1_{\et}(Y(n)/\overline{\mathbf Q},\calV^{2m}_{\mathbf Z_p}(2m+2)))
	\to \mathrm H^1(\mathbf Q(\mu_n),T).
\end{multline*}
Here, $e_{1,np^r},e_{2,np^r}$ are universal torsion points of the universal elliptic
curve over $Y(np^r)$.
We also define $z_n(i,j)\in\mathbf H^1(\mathbf Q(\mu_n),T\otimes_{\mathbf Z}\mathbf Q)$
by
\[
z_n(i,j)=(c^2-c^{-2i})(d^2-d^{-2j}){_{c,d}z_n(i,j)}.
\]
\begin{theorem}[{\cite[Proposition 8.10]{Kat04}}]
	The set $\{{_{c,d}z_n(i,j)}\}_{n\in \Xi(6pcd)}$ forms an Euler system for $(T,6pcd)$.
	\label{thm4-3}
\end{theorem}
\begin{proof}
Let $I_2$ denote the two-by-two identity matrix. 
Then, by construction, the element $_{c,d}z_n(i,j)$ coincides with
$_{c,d}z_{1,1,n}(2m+2,0,2i-1,I_2,\mathrm{Prime}(n))$ which is defined in \cite[Section 8.9]{Kat04}.
Then, the assertion of the theorem is a direct consequence of \cite[Proposition 8.10]{Kat04}.
\end{proof}
Similar as before, we define ${_{c,d}z_1(i,j)}$ and ${z_1(i,j)}$
to be corestriction of ${_{c,d}z_p(i,j)}$ and ${z_p(i,j)}$
to $\mathbf Q$, respectively.
As a direct consequence of the theorem above and the Euler system argument,
we have the following corollary:
\begin{corollary}\label{cor4-4}
	Let $f$ be a full-level Hecke eigen cuspform of weight $2m+2$
	and let $e_f$ denote the idempotent of the Hecke algebra
	corresponding to $f$.
	If the element
	\[
	e_fz_1(i,j)\in \mathrm H^1\left(\mathbf Z\left[\frac{1}{p}\right],V_f(2m+2)\right)
	\]
	is non-zero, then we have
	\[
	\mathrm H^2\left(\mathbf Z\left[\frac{1}{p}\right],V_f(2m+2)\right)=0.
	\]
\end{corollary}
\begin{proof}The first condition of Theorem \ref{thm4-1}
	is a direct consequence of the Weil conjecture proved by Deligne.
	Since the level of $f$ is full, this modular form has no CM.
	Therefore, by the big Galois image theorem by Ribet,
	the second condition in Theorem \ref{thm4-1} for any lattice of $V_f$
	is satisfied.
	Thus, by Theorem \ref{thm4-1}, we have that $\rmH^2(\bfZ[1/6cdp],V_f(2m+2))=0$
	for any $c,d,$ which are coprime to $6p$.
	By the localizing exact sequence, we have the following exact sequence:
	\[
\bigoplus_{\ell\mid 6cd,\ \ell\nmid p}\rmH^1(\bfF_\ell,V_f(-1))\to 	\rmH^2(\bfZ[1/p],V_f(2m+2))\to\rmH^2(\bfZ[1/6cdp],V_f(2m+2))=0.
	\]
	Then, the first cohomology groups vanish because the homomorphism $\mathrm{Fr}_\ell-1$ on $V_f(-1)$ is an isomorphism by the Weil conjecture. Therefore, we have the conclusion of the corollary.
\end{proof}
\subsection{Non-vanishing of Euler systems and proof of the main theorem}
In this section, we give a proof of our main result
assuming relation between specializations of Kato Euler systems
and cup products of Eisenstein classes.
\begin{theorem}
	\label{thm4-5}Let $i,j$ be positive integers such that $m=i+j$.
	Then, we have
	\[
	\left(1-T_p+p^{2m+1}\right)\left(\varphi_{2i+2}\cup\varphi_{2j+2}\right)=z_1(i,j),
	\]
	where $T_p$ is the $p$th Hecke operator.
\end{theorem}
Before to show the theorem, let us give a proof of our main theorem.
We will give a proof of Theorem \ref{thm4-5} in the next subsection.
\begin{corollary}\label{cor4-6}
	Under Hypothesis $\ref{hyp2-2}$, the dimension of the $\mathbf Q_p$-vector space
	\[
	\sum_{i+j=m}\mathbf Q_pz_1(i,j)\subset \mathrm H^1_{\mathrm f}(\mathbf Q,\mathrm H^1_{\mathrm{cusp}}({\scrM_{1,1}}/\overline{\mathbf Q},\calV^{2m}(2m+2))
	\]
	is greater than or equal to the dimension of $S_{2m+2}$.
\end{corollary}
\begin{proof}
	By the Weil conjecture proved by Deligne, $\left(1-T_p+p^{2m+1}\right)$
	is an invertible operator on $\mathrm H^1_{\mathrm f}(\mathbf Q,\mathrm H^1_{\mathrm{cusp}}({\scrM_{1,1}}/\overline{\mathbf Q},\calV^{2m}(2m+2))$. Therefore, the conclusion
	of this corollary is a direct consequence of Theorem \ref{thm4-5}
	and Corollary \ref{cor3-6}.
\end{proof}
\begin{theorem}
	\label{thm4-7}Under Hypothesis \ref{hyp2-2}, the following equation of $\mathbf Q_p$-vector spaces
	hold:
	\[
	\sum_{i+j=m}\mathbf Q_pz_1(i,j)= \mathrm H^1_{\mathrm f}(\mathbf Q,\mathrm H^1_{\mathrm{cusp}}({\scrM_{1,1}}/\overline{\mathbf Q},\calV^{2m}(2m+2)).
	\]
	Moreover, we have
	\[
	\dim_{\mathbf Q_p}\mathrm H^1_{\mathrm f}(\mathbf Q,\mathrm H^1_{\mathrm{cusp}}({\scrM_{1,1}}/\overline{\mathbf Q},\calV^{2m}(2m+2))=\dim_{\mathbf Q}S_{2m+2}.
	\]
\end{theorem}
\begin{proof}Let $\calB_{2m+2}={f_1,\dots,f_n}$ be the set of full-level normalized Hecke eigen cuspforms
	of weight $2m+2$ such that the equation
	\[
	\mathrm H^1_{\mathrm{cusp}}({\scrM_{1,1}}/\overline{\mathbf Q},\calV^{2m})=\bigoplus_{i=1}^nV_{f_i}
	\] holds.
According to Corollary \ref{cor3-6}, we may assume that the natural projection
\[
\sum_{i+j=m}\mathbf Q_pz_1(i,j)\subset \mathrm H^1_{\mathrm f}(\mathbf Q,\mathrm H^1_{\mathrm{cusp}}({\scrM_{1,1}}/\overline{\mathbf Q},\calV^{2m}(2m+2))\ontomap \mathrm H^1_{\mathrm f}(\mathbf Q,V_{f_1}(2m+2))
\]
is non-zero. This implies that there exists a positive integers $i$ and $j$ such that
$i+j=2m $ and $e_{f_1}z_1(i,j)\neq 0$.
Therefore, by Corollary \ref{cor4-6} and by the Euler-Poincare characteristic of
Galois cohomology (\cite[Lemma 2]{Jannsen89}), we have that
\begin{equation}\label{eq22}
	\dim_{\bfQ_p} \mathrm H^1_{\mathrm f}(\mathbf Q,V_{f_1}(2m+2))=
\dim_{\bfQ_p} \mathrm H^1(\mathbf Z[1/p],V_{f_1}(2m+2))=\frac{1}{2}\dim_{\bfQ_p}(V_{f_1}).
\end{equation}
Here, the first equation  follows from \cite[Corollary 3.8.4]{BK90} and from the equation $F^0D_{\dR}(V_{f_i}(2m+2))=0$.
Since the equation
\begin{equation}
	\label{eq23}
\dim_{\bfQ}(S_{2m+2})=\frac{1}{2}\dim_{\bfQ_p}(\mathrm H^1_{\mathrm{cusp}}({\scrM_{1,1}}/\overline{\mathbf Q},\calV^{2m}))=\sum_{i=1}^n\frac{1}{2}\dim_{\bfQ_p}(V_{f_i})
\end{equation}
holds by the Eichler-Shimura isomorphism, 
the first assertion of the Theorem holds when $n=1$.

If $n\geq 2$, we repeat the same argument as follows.
Let $W\subset  \bigoplus_{i=2}^n\rmH^1(\bfQ,V_{f_i}(2m+2))$ be the image of $\sum_{i+j=m}\mathbf Q_pz_1(i,j)$ under the natural projection
\[
 \mathrm H^1_{\mathrm f}(\mathbf Q,\mathrm H^1_{\mathrm{cusp}}({\scrM_{1,1}}/\overline{\mathbf Q},\calV^{2m}(2m+2))\ontomap\bigoplus_{i=2}^n\rmH^1(\bfQ,V_{f_i}(2m+2)).
\]
Then, by the equations (\ref{eq22}), (\ref{eq23}), and Corollary \ref{cor3-6},
the inequality
\[
\dim_{\bfQ_p}(W)\geq \sum_{i=2}^n\frac{1}{2}\dim_{\bfQ_p}(V_{f_i})
\]
holds. Therefore, $W\neq 0$ if $n\geq 2$. Then, by repeating the same argument in the first paragraph, we can show that
\begin{equation}\label{eq24}
	\dim_{\bfQ_p} \mathrm H^1_{\mathrm f}(\mathbf Q,V_{f_i}(2m+2))=\frac{1}{2}\dim_{\bfQ_p}(V_{f_i}).
\end{equation}
for all $i$ by an induction on $n$.
By (\ref{eq22}), (\ref{eq24}), and Corollary \ref{cor3-6},
we have the conclusion of the theorem.
\end{proof}

\subsection{Eisenstein symbols and the proof of Theorem \ref{thm4-5}}
To prove Theorem \ref{thm4-5}, we use Beilinson's Eisenstein symbols.
Let $G$ denote $\mathrm{GL}_2$ over $\mathbf Z$ and let $P$
be the subgroup of $G$ defined by
\[
P=\left\{\begin{pmatrix}
	*&*\\
	0&1
\end{pmatrix}\in G\right\}.
\]
Let $I_2$ be the two-by-two identity matrix. For a commutative ring $R$, let $\la -I_2\ra P(R)$ be the subgroup of $G(R)$ generated by $P(R)$ and $-I_2$.
Let $\mathrm{Ind}_{\la -I_2\ra P(\mathbf Z/n\mathbf Z)}^{G(\mathbf Z/n\mathbf Z)}(\mathbf Q)$ be the induced representation of the trivial representation of
$\la -I_2\ra P(\bfZ/n\bfZ)$ on $\bfQ$ to $G(\bfZ/n\bfZ)$. Explicitly, the underlying vector space
of this representation is given by
\begin{multline}
\mathrm{Ind}_{\la -I_2\ra P(\mathbf Z/n\mathbf Z)}^{G(\mathbf Z/n\mathbf Z)}(\mathbf Q)\\
=\{f\colon G(\bfZ/n\bfZ)\to \bfQ\ |\ f(hg)=f(g),\ \forall g\in G(\bfZ/n\bfZ),\ \forall h\in \la -I_2\ra P(\bfZ/n\bfZ)\}.
\end{multline}
The universal elliptic curve over $Y(n)$ is denoted by
\[
\pi_n\colon \mathscr E_n\to Y(n)
\]
and let $\mathscr E_n^k$ be the $k$-fold fiber product
of $\mathscr E_n$ over $Y(n)$.
We only consider the case where \emph{$k$ is even} for simplifying notation.
For each positive integers $m,n$ such that $m|n$, $\bp_{n,m}\colon Y(n)\to Y(m)$
be the natural morphism defined by $[E,\alpha]\mapsto [E,n/m\alpha]$.
When $m=1$, $\bp_{n,1}$ will be denoted by $\bp_{n}$.
The symbol $X(n)$ denotes the smooth compactification of $Y(n)$.

For an algebraic variety $S$ over $\mathbf Q$,
its motivic cohomology group $\mathrm H^i_{\mathcal M}(S,\mathbf Q(j))$
to be $K_{2j-i}^{(j)}(S)\otimes_{\mathbf Z}\mathbf Q$.
The motivic residue map is a homomorphism
\begin{equation}
	\mathrm{res}_{\mathcal M}\colon \mathrm H^{k+1}_{\mathcal M}(\mathscr E_n^k,\mathbf Q(k+1))\to \mathrm H^0_{\mathcal M}\left(X(n)\setminus Y(n),\mathbf Q\right)
	\cong \mathrm{Ind}_{\la -I_2\ra P(\mathbf Z/n\mathbf Z)}^{G(\mathbf Z/n\mathbf Z)}(\mathbf Q)
\end{equation}
which is compatible with realizations.
Here, for the last isomorphism, we use the same uniformization of modular curves as in \cite{HuK99}, which is different from that in \cite{Kat04}.
In \cite{Beilinson86}, Beilinson constructed a homomorphism
\begin{equation}
	\label{eq18}
	\mathrm{Eis}_k\colon \mathrm{Ind}_{\la -I_2\ra P(\mathbf Z/n\mathbf Z)}^{G(\mathbf Z/n\mathbf Z)}(\mathbf Q)\to
	\mathrm H^{k+1}_{\mathcal M}(\mathscr E_n^k,\mathbf Q(k+1))
\end{equation}
which is a right-inverse of $\mathrm{res}_{\mathcal M}$.
We define $\mathrm{Eis}_k^p$ to be the composition of $\mathrm{Eis}_k$
, the Soul\'e's $p$-adic \'etale regulator,
and projection by Scholl's projector $\Pi_{\epsilon}$ (\cite[1.1.2]{Scholl})
\begin{multline}
	\mathrm{Eis}_k^p\colon \mathrm{Ind}_{\la -I_2\ra P(\mathbf Z/n\mathbf Z)}^{G(\mathbf Z/n\mathbf Z)}(\mathbf Q)\xrightarrow{\mathrm{Eis}_k}\mathrm H^{k+1}_{\mathcal M}(\mathscr E_n^k,\mathbf Q(k+1))\\
	\xrightarrow{\reg_{\pet}}\mathrm H^{k+1}_{\et}(\mathscr E_n^k,\mathbf Q_p(k+1))\to \mathrm H^1_{\et}(Y(n),\calV_k(k+1)).
\end{multline}
Note that, when $k$ is not equal to zero, then the $\mathbf Q_p$-linear extension of $\mathrm{Eis}_k^p$
is an isomorphism.
Let
\[
u_n\colon \frakH\times \bfC^k\times G(\bfZ/n\bfZ)\twoheadrightarrow (\bfZ^{2k}\rtimes \SL_2(\bfZ))\backslash(\frakH\times \bfC^k\times (\bfZ/n\bfZ)^k\rtimes G(\bfZ/n\bfZ))\cong \scrE_n^k(\bfC)
\]
be a complex uniformization of $\scrE_n^k$ as in \cite[Section 7, p.328]{HuK99}.
Let $\tau$ and $z_i$ be the standard coordinates of $\frakH$ and the $i$th component
of $\bfC^k$, respectively.
Then,  the pull-back of $\bp_n^*\varphi_{k+2}\in  \rmH^k(\scrE_n^k(\bfC),\bfQ(k+1))$ to $\frakH\times \bfC^k\times \{g\}$ is represented by the differential form
\[
(2\pi\sqrt{-1})^k\left(-\frac{B_{k+2}}{k+2}+\sum_{n=1}^\infty\sigma_{k+1}(n)q^n\right)\frac{dq}{q}\wedge dz_1\wedge\cdots\wedge dz_k
\]
for each $g\in G(\bfZ/n\bfZ)$,
where $B_{k+2}$ is the $(k+2)$nd Bernoulli number and $q$ is equal to $\exp(2\pi\sqrt{-1}\tau)$.
Therefore, by comparing an explicit description of the Eisenstein symbol in \cite[p.329]{HuK99},
we have the equation that
\begin{equation}
	\label{eq19}
	\boldsymbol p_n^*(\varphi_{k+2})=\mathrm{Eis}_k^{p}\left(-n\frac{B_{k+2}}{k+2}\right).
\end{equation}
Here, we understand each rational number $\alpha$ as the constant function defined by $\alpha$, which is an element of $\mathrm{Ind}_{\la -I_2\ra P(\mathbf Z/n\mathbf Z)}^{G(\mathbf Z/n\mathbf Z)}(\mathbf Q)$.
\begin{definition}
	\label{dfn4-7}We define two elements $\phi_1^{k,n},\phi_2^{k,n}$
	of $\mathrm{Ind}_{\la -I_2\ra P(\mathbf Z/n\mathbf Z)}^{G(\mathbf Z/n\mathbf Z)}(\mathbf Q)$.
		\[
	\phi^{k,n}_1\left(\begin{pmatrix}
		a&b\\c&d
	\end{pmatrix}\right)=\frac{n^{k+1}}{k+2}B_{k+2}\left(\left\langle\frac{c}{n} \right\rangle\right)
	\]
	and by
	\[
	\phi^{k,n}_2\left(\begin{pmatrix}
		a&b\\c&d
	\end{pmatrix}\right)=\frac{n^{k+1}}{k+2}B_{k+2}\left(\left\langle\frac{d}{n}\right \rangle\right).
	\]
	Here, $\langle -\rangle$ is the inverse of the natural bijection $[0,1)\xrightarrow{\sim}\mathbf R/\mathbf Z$ and $B_{k+2}(x)$ is the $(k+2)$nd Bernoulli polynomial.
\end{definition}
\begin{proposition}[{\cite[Proposition 10.2.1, Lemma 10.3.1]{Gealy}}]
	\label{prop4-8}The following equation holds$:$
	\[
	\mathrm{Eis}_{2i}^p(\phi^{2i,p}_1)\cup \mathrm{Eis}_{2j}^p(\phi^{2j,p}_2)=z_p(i,j).
	\]
\end{proposition}
\begin{remark}
	Note that those two maps are different from $\phi_1,\phi_2$ defined in {\cite[Subsection 5.3]{Gealy}}
	visually.
	However, the Proposition \ref{prop4-8} above still holds. This difference comes from the difference of the uniformization of modular curves.
\end{remark}
In the rest of this paper, we consider the case $n=p$.
Let
\[
\bq\colon \mathscr E_p^k\to\mathscr E^k
\]
be the induced morphism by $\boldsymbol p=\boldsymbol p_p$.
Then, Theorem \ref{thm4-5} is a direct consequence of Proposition \ref{prop4-8} and the following proposition.
\begin{proposition}
	\label{prop4-9}We have
	\[
	\bq_{*}\left(\mathrm{Eis}_{2i}^p(\phi^{2i,p}_1)\cup \mathrm{Eis}_{2j}^p(\phi^{2j,p}_2)\right)=\left(1-T_p+p^{2m+1}\right)\left(\varphi_{2i+2}\cup\varphi_{2j+2}\right).
	\]
\end{proposition}
Though the proposition above seems to be well-known for professional
researchers, we give a proof of this because the author could not find
a suitable reference (it seems that the proof of \cite[Lemma 10.3.2]{Gealy} is incompleted).
For a proof of this proposition, we prepare some lemmas.
Let $Y_1(p)$ and $ Y_0(p)$ be the moduli stacks over $\bfQ$ of
elliptic curves with $\Gamma_1(p)$-level structures and $\Gamma_0(p)$-level structures,
respectively (\cite[(3.2), (3.4)]{KM84}).
Let
\[
Y(p)\to Y_1(p)\to Y_0(p)\to\scrM_{1,1} 
\]
be a natural factorization of $\boldsymbol p$ defined by
\[
[E,e_1,e_2]\mapsto [E,e_2]\mapsto [E,\mathbf Ze_2]\mapsto [E].
\]
Those morphisms of moduli stacks also induce a factorization
\[
\mathscr E_p^k\xrightarrow{{\boldsymbol q}_1}\mathscr E_{1,p}^k\xrightarrow{{\boldsymbol q}_2}\mathscr E_{0,p}^k\xrightarrow{{\boldsymbol q}_3}\mathscr E^k
\]
of $\bq$.
Here, $\mathscr E_{i,p}$ denotes the universal elliptic curve over $Y_i(p)$.
Let
$
\lambda_p\colon Y_0(p)\xrightarrow{\sim}Y_0(p)
$
be the morphism defined by the correspondence
\[
\lambda_p;\ [E,C]\mapsto \left[E/C,\text{Image of } E[p]\right]
\]
and let
\[
\lambda_p^k\colon \mathscr E_{1,p}^k\to \mathscr E_{0,p}^k
\]
be the induced isogeny of degree $p^k$.
Note that, under the complex uniformization
\[
(\bfZ^{2k}\rtimes\SL_2(\bfZ))\backslash (\frakH\times \bfC^k\times \bfF_p^{2k}\rtimes G(\bfF_p))/B(\bfF_p)\cong \mathscr E_{0,p}^k(\bfC),
\]
the morphism $\lambda_p^k$ is given by
\begin{equation}
	\label{descriptionoflambdap}
	\lambda_p^k;\ [\tau,\bz,\bu,g]\mapsto \left[\frac{-1}{p\tau},\frac{-1}{\tau}\bz,\bu,g \right].
\end{equation}
Here, $B$ is the Borel subgroup of $G$ containing $P$.
Let $\tilde{\phi}$ be an element of the module $\mathrm{Ind}_{\la -I_2\ra P(\mathbf F_p)}^{G(\mathbf F_p)}\mathbf Q$ defined by the equation
\[
\tilde \phi_k\left(\begin{pmatrix}
	a&b\\c&d
\end{pmatrix}\right)=-\frac{p^{k+1}}{k+2}\sum_{\alpha\in \mathbf F_p}B_{k+2}\left(\left\langle\frac{\alpha c}{p}\right\rangle\right).
\]
\begin{lemma}
	\label{lem4-10}
	We have that
	\[
	\mathrm{Eis}_k^p(\tilde \phi_k)=({\boldsymbol q}_2{\boldsymbol q}_1)^*{\lambda}_p^{k,*}\bq_3^*\varphi_{k+2}.
	\]
\end{lemma}
\begin{proof}
	Let
	\begin{multline}
	\mathrm{Eis}_k^{\mathcal D}\colon \mathrm{Ind}_{\la -I_2\ra P(\mathbf F_p)}^{G(\mathbf F_p)}\mathbf Q\to \mathrm H^1_{\mathcal D}\left({Y(p)}/\mathbf R,\mathrm{Sym}^k\rmR^1\pi_*\mathbf R(k+1)\right)\\
	\intomap \rmH^1(Y(p)(\bfC),\mathrm{Sym}^k\rmR^1\pi_*\mathbf C)
	\end{multline}
	be the composition of the Eisenstein symbol, the Beilinson regulator
	and the natural projection.
	Since $\mathrm{Eis}_k^{\mathcal D}$ is injective, to prove this lemma, it is sufficient to show
	the equation
	\[
	\mathrm{Eis}_k^{\mathcal D}(\tilde \phi)=\mathrm{Eis}_k^{\mathcal D}(\mathrm{Eis}_k^{p,-1}(	\tilde{\lambda}_p^*\bq_3^*\varphi_{k+2})).
	\]
	
	By the explicit description (\ref{descriptionoflambdap}) of $\lambda_p^k$,
	a $(k+1)$-form on $\frakH\times\bfC^k$ representing the class ${\lambda}_p^{k,*}\bq_3^*\varphi_{k+2}$ is computed as follows:
	\begin{equation}
		\begin{split}
		(\lambda_p^k)^*&\left((2\pi\sqrt{-1})^{k+1}\frac{-B_{k+2}p^{k+1}}{k+2}\sum_{(c,d)=1}\frac{1}{(c\tau+d)^{k+2}}d\tau\wedge dz_1\wedge\cdots\wedge dz_k\right)\\
		&=(2\pi\sqrt{-1})^{k+1}\frac{-B_{k+2}p^{k+1}}{k+2}\sum_{(c,d)=1}\frac{1}{\left(c\frac{-1}{p\tau}+d\right)^{k+2}}d\left(\frac{-1}{p\tau}\right)\wedge d\left(\frac{-1}{\tau}z_1\right)\wedge\cdots\wedge d\left(\frac{-1}{\tau}z_k\right)\\
		&=(2\pi\sqrt{-1})^k\frac{-B_{k+2}p^{k+1}}{k+2}\sum_{(c,d)=1}\frac{1}{(pc\tau+d)^{k+2}}\frac{dq}{q}\wedge dz_1\wedge\cdots\wedge dz_k\\
		&=(2\pi\sqrt{-1})^k\frac{-B_{k+2}}{p(k+2)}\sum_{(c,d)=1,p\nmid c}\frac{1}{(c\tau+d)^{k+2}}\frac{dq}{q}\wedge dz_1\wedge\cdots\wedge dz_k\\
		&\hspace{1cm}+(2\pi\sqrt{-1})^k\frac{-B_{k+2}p^{k+1}}{k+2}\sum_{(c,d)=1,\ p|c}\frac{1}{(c\tau+d)^{k+2}}\frac{dq}{q}\wedge dz_1\wedge\cdots\wedge dz_k.
		\end{split}
	\end{equation}
	On the other hand, by	\cite[p.329]{HuK99},
	a $(k+1)$-form on $\frakH\times \bfC^k=\frakH\times \bfC^k\times \{I_2\}$ representing $\mathrm{Eis}_k^p(\tilde \phi_k)$
	is computed as follows:
	\begin{equation}
	\begin{split}
		\mathrm{Eis}_k^p(\tilde \phi_k)&=\frac{(2\pi\sqrt{-1})^k}{p}\sum_{(c,d)=1}\frac{1}{(c\tau+d)^{k+2}}\tilde{\phi}_k\left(\begin{pmatrix}
			*&*\\
			c&d
			\end{pmatrix}\right)\frac{dq}{q}\wedge dz_1\wedge\cdots\wedge dz_k\\
			&=\frac{(2\pi\sqrt{-1})^k}{p}\sum_{(c,d)=1}\frac{1}{(c\tau+d)^{k+2}}\left(-\frac{p^{k+1}}{k+2}\sum_{\alpha\in \mathbf F_p}B_{k+2}\left(\left\langle\frac{\alpha c}{p}\right\rangle\right)\right)\\
			&\hspace{5cm}\times \frac{dq}{q}\wedge dz_1\wedge\cdots\wedge dz_k\\
			&=\frac{(2\pi\sqrt{-1})^k}{p}\sum_{(c,d)=1,p\mid c}\frac{1}{(c\tau+d)^{k+2}}\left(-\frac{p^{k+2}}{k+2}B_{k+2}\right) \frac{dq}{q}\wedge dz_1\wedge\cdots\wedge dz_k\\
			&+\frac{(2\pi\sqrt{-1})^k}{p}\sum_{(c,d)=1,p\nmid c}\frac{1}{(c\tau+d)^{k+2}}\left(-\frac{p^{k+1}}{k+2}\sum_{\alpha=1}^pB_{k+2}\left(\frac{\alpha}{p}\right)\right)\\
			&\hspace{5cm}\times \frac{dq}{q}\wedge dz_1\wedge\cdots\wedge dz_k.
	\end{split}
	\end{equation}
	Then, by the equation
	\[
	\sum_{\alpha=1}^pB_{k+2}\left(\frac{\alpha}{p}\right)=p^{-(k+1)}B_{k+2}
	\]
	(\cite[Proposition 4.1]{Washington}), we have the conclusion of the lemma.
\end{proof}
\begin{lemma}
	\label{lem4-11}The following equation holds$:$
	\[
	{\boldsymbol q}_{3,*} \lambda_p^*\bq_3^*\varphi_{k+2}=(1+p^{k+1})\varphi_{k+2}.
	\]
\end{lemma}
\begin{proof}According to \cite[Lemma 4.5.15]{Miyake},
	the $p$th Hecke eigen value of normalized Hecke eigen form coincides with its $p$th Fourier coefficient. Therefore, we have the conclusion of the lemma because the $p$th Fourier coefficient of the weight $k$ normalized Eisenstein series for $\SL_2(\bfZ)$ is equal to $\sigma_{k-1}(p)=1+p^{k-1}$.
\end{proof}
\begin{proof}[Proof of Proposition $\ref{prop4-9}$]
	To compute $\bq_*={\boldsymbol q}_{3,*}{\boldsymbol q}_{2,*}{\boldsymbol q}_{1,*}$,
	we compute each ${\boldsymbol q}_{i,*}$ step by step.
	
	First, we compute ${\boldsymbol q}_{1,*}$. By the projection formula, the equation
	\[
	{\boldsymbol q}_{1,*}\left(\mathrm{Eis}_{2i}^p(\phi^{2i,p}_1)\cup \mathrm{Eis}_{2j}^p(\phi^{2j,p}_2)\right)=\mathrm{Eis}_{2i}^p(\phi^{2i,p}_1)\cup {\boldsymbol q}_{1,*}\mathrm{Eis}_{2j}^p(\phi^{2j,p}_2)
	\]
	holds. Then, we have that \[
	{\boldsymbol q}_{1,*}\mathrm{Eis}_{2j}^{2j,p}(\phi^p_2)=-({\boldsymbol q}_2{\boldsymbol q}_3)^*(\varphi_{2j+2})+{\boldsymbol q}_2^*\lambda_p^{k,*}\bq_3^*\varphi_{2j+2}.
	\]
	This equation follows from equations
	\begin{equation*}
		\begin{split}
			\mathrm{Eis}_{2j}^{p,-1}{\boldsymbol q}_1^*{\boldsymbol q}_{1,*}\mathrm{Eis}_{2j}^p(\phi^{2j,p}_2)(g)&=\frac{p^{2j+1}}{2j+2}\sum_{\alpha\in \mathbf F_p^\times,\beta\in \mathbf F_p}B_{2j+2}\left(\left\langle\frac{\alpha c+\beta d}{p}\right\rangle\right)\\
			&=\frac{p^{2j+1}}{2j+2}\left(\sum_{\alpha,\beta\in \mathbf F_p}B_{2j+2}\left(\left\langle\frac{\alpha c+\beta d}{p}\right\rangle\right)
			-\sum_{\beta\in \mathbf F_p}B_{2j+2}\left(\left\langle\frac{\beta d}{p}\right\rangle\right)\right)\\
			&=\frac{pB_{2j+2}}{2j+2}+\tilde \phi_{2j}(g),
		\end{split}
	\end{equation*}
	where $g=\begin{pmatrix}
		a&b\\
		c&d
	\end{pmatrix}$.
	
	Next, we compute ${\boldsymbol q}_{2,*}$.
	By a similar computation as above and by Lemma \ref{lem4-10} again,
	we have the following equation:
	\begin{multline}
		\label{eq21}
		{\boldsymbol q}_{2,*}\left(\mathrm{Eis}_k^{p}(\phi_1^{2i,p})\cup(-({\boldsymbol q}_2{\boldsymbol q}_3)^*(\varphi_{2j+2})+{\boldsymbol q}_2^*\lambda_p^{k,*}\bq_3^*\varphi_{2j+2})\right)\\
		={\boldsymbol q}_{2,*}\mathrm{Eis}_k^{p}(\phi_1^{2i,p})\cup(-{\boldsymbol q}_3^*(\varphi_{2j+2})+\lambda_p^{k,*}\bq_3^*\varphi_{2j+2})\\
		=(p^{2i}\bq_3^*(\varphi_{2i+2})-\lambda_p^*\bq_3^*\varphi_{2i+2})\cup(-{\boldsymbol q}_3^*(\varphi_{2j+2})+\lambda_p^{k,*}\bq_3^*\varphi_{2j+2}).
	\end{multline}
	
	Finally, we compute the push-forward by ${\boldsymbol q}_3$.
	This is computed by using Lemma \ref{lem4-11} as follows:
	\begin{multline}
		{\boldsymbol q}_{3,*}\left((p^{2i}\bq_3^*(\varphi_{2i+2})-\lambda_p^*\bq_3^*\varphi_{2i+2})\cup(-{\boldsymbol q}_3^*(\varphi_{2j+2})+\lambda_p^{k,*}\bq_3^*\varphi_{2j+2})\right)\\
		=-\bq_{3,*}(p^{2i}\bq_3^*(\varphi_{2i+2})\cup{\boldsymbol q}_3^*(\varphi_{2j+2}))+\bq_{3,*}(p^{2i}\bq_3^*(\varphi_{2i})\cup \lambda_p^{k,*}\bq_3^*\varphi_{2j+2}))\\
		+\bq_{3,*}(\lambda_p^*\bq_3^*\varphi_{2i+2}\cup {\boldsymbol q}_3^*(\varphi_{2j+2}))+\bq_{3,k*}(-\lambda_p^*\bq_3^*\varphi_{2i+2}\cup)\lambda_p^{k,*}\bq_3^*\varphi_{2j+2}\\
		=-(1+p)p^{2i}\varphi_{2i+2}\cup \varphi_{2j+2}+p^{2i}(1+p^{2j+1})\varphi_{2i+2}\cup \varphi_{2j+2}\\
		+(1+p^{2i+1})\varphi_{2i+2}\cup \varphi_{2j+2}-T_p\varphi_{2i+2}\cup \varphi_{2j+2}\\
		=(1-T_p+p^{2m+1})(\varphi_{2i+2}\cup \varphi_{2j+2}).
	\end{multline}
	This completes the proof of Proposition \ref{prop4-9}.
\end{proof}
\appendix
\section{Review of $p$-adic weight completions and basic examples}
In this appendix, we give a quick review of the $p$-adic weighted completion introduced by Hain-Matsumoto in \cite{HM03}.
Our explanation is based on the paper \cite{HM04}.
\subsection{General theory}\label{generaltheory}
Let $S$ be a reductive algebraic group over $\bfQ_p$ and suppose that we are given a central cocharacter
$w\colon \bfG_m\to S$.
For each $S$-module $V_0$, the symbol $w^*V_0$ denotes the $\bfG_m$-module $V_0$
defined by $w$ and the original action of $S$ on $V_0$.
Let $\pi$ be a profinite group and let
\[
\rho\colon \pi\to S(\bfQ_p)
\]
be a continuous homomorphism whose image is Zariski dense.
A \emph{weighted module with respect to $(\rho,w)$} is a pair $(V,W_\bullet V)$ where
\begin{itemize}
	\item $V$ is a finite dimensional $\bfQ_p$-vector space equipped with a continuous action of $\pi$,
	\item $W_\bullet V$ is a separated and saturated increasing filtration of $\pi$-modules indexed by $\bfZ$,
\end{itemize}
satisfying the following conditions:
\begin{itemize}
	\item For each $n$, the $\pi$-module $\Gr^W_nV=W_nV/W_{n-1}V$ is the pull-back of an algebraic representation of $S$ via $\rho$. Here, for an algebraic representation $\tau$ of $S$, the pull-back of $\tau$ via $\rho$ is defined to be $\tau\circ \rho$.
	\item For each $n$, the $\bfG_m$-module $w^*\Gr^W_nV$ is a direct sum of the $n$th power $\std_{\bfG_m}^{\otimes n}$ of the standard character $\std_{\bfG_m}$ of $\bfG_m$.
\end{itemize}
Let $\Rep_{\bfQ_p}(\pi,\rho,w)$ denote the category of weighted $\pi$-modules with respect to $(\rho,w)$.
It is easily checked that this category is a $\bfQ_p$-linear neutral Tannakian category and that the forgetful functor
\[
\omega_{\mathrm f}\colon \Rep_{\bfQ_p}(\pi,\rho,w)\to \Vec_{\bfQ_p};\quad (V,W_\bullet V)\mapsto V
\]
is a fiber functor of this category.
The \emph{$p$-adic weighted completion of $\pi$ with respect to $(\rho,w)$}, which is denoted by $\pi^{\wtd}(\rho,w)$,
is defined to be the Tannakian fundamental group of $\Rep_{\bfQ_p}(\pi,\rho,w)$ with the fiber functor $\omega_{\mathrm f}$.
By definition, we have a short exact sequence
\[
1\to \pi^{\wtd,\unip}(\rho,w)\to \pi^{\wtd}(\rho,w)\to S\to 1
\]
of proalgebraic groups over $\bfQ_p$, where $\pi^{\wtd,\unip}(\rho,w)$ is the prounipotent radical of $\pi^{\wtd}(\rho,w)$.
Therefore, the full-subcategory of $\Rep_{\bfQ_p}(\pi,\rho,w)$ consisting of semi-simple
objects is canonically equivalent to the category $\Rep_{\bfQ_p}(S)$ of algebraic representations of $S$
on finite dimensional $\bfQ_p$-vector spaces.

It is easily checked that any morphism $f\colon (V,W_\bullet V)\to (V',W_\bullet V')$ is strict compatible
with the weight filtrations. Therefore, the functor
\[
\Gr^W_\bullet\colon \Rep_{\bfQ_p}(\pi,\rho,w)\to \Rep_{\bfQ_p}(\pi,\rho,w);\quad (V,W_\bullet V)\mapsto \Gr^W_\bullet(V)
\]
is an exact functor.

The induced representation $\rho^{\univ}\colon \pi\to \pi^{\wtd}(\rho,w)(\bfQ_p)$ by the natural functor
\begin{equation}
	\label{naturalfc}
	\Rep_{\bfQ_p}(\pi,\rho,w)\to \Rep_{\bfQ_p}(\pi);\quad (V,W_\bullet V)\mapsto V,
\end{equation}
where the latter is the category of continuous representations of $\pi$ on finite dimensional $\bfQ_p$-vector spaces,
is called the \emph{universal representation}.
According to (\cite[Proposition 7.2]{HM04}), $\rho^{\univ}$ has a Zariski-dense image.
Therefore, the natural functor (\ref{naturalfc}) is fully-faithful.

For a proalgebraic group $G$ over $\bfQ_p$ and a finite dimensional algebraic representation $V$ of $G$, $\rmH^i(G,V)$
denotes the Yoneda extension group $\Ext^i_{\Rep_{\bfQ_p}(G)}(\bfQ_p,V)$ of the category $\Rep_{\bfQ_p}(G)$ of finite dimensional algebraic representations of $G$. The following is the fundamental result of \cite{HM03}:
\begin{theorem}[{\cite[Theorem 8.1]{HM04}}]Let $V$ be a simple $S$-module of weight $n$ with respect to $w$.
	Let \begin{equation}\label{comptcoh}
		\rmH^i(\pi^{\wtd}(\rho,w),V)\to \rmH^i_{\mathrm{cont}}(\pi,V)
	\end{equation}
	be the natural homomorphism induced by the universal representation $\rho^{\univ}$.
	\begin{itemize}
		\item[(1)]If $n$ is negative, then (\ref{comptcoh}) is an isomorphism for $i=1$.
		\item[(2)]For $i=2$, the homomorphism (\ref{comptcoh}) is injective.
	\end{itemize}
\end{theorem}
Of course, (\ref{comptcoh}) is isomorphism when $i=0$ by the Zariski density of the image
of $\rho^{\univ}$.

\subsection{Mixed Tate modules}\label{appendix2}
In this and the next subsection, we see basic examples of weighted completions
which are used in this paper.

Let $K$ be a finite number field and let
\[
\chi_{\cyc,p}\colon \scrG_K\to \bfQ_p^\times=\bfG_m(\bfQ_p)
\]
be the $p$-adic cyclotomic character.
Let us take the central character $w\colon \bfG_m\to \bfG_m$ to be
$\std^{-2}_{\bfG_m}$.
Then, the category $\Rep_{\bfQ_p}(G_K,\chi_{\cyc,p},\std^{-2}_{\bfG_m})$
is called the category of \emph{$p$-adic Tate modules over $K$} (\cite[Section 6]{HM04}).
Let $\MT(K)$ be the category of mixed Tate motives over $K$.
It is well-known that the $p$-adic \'etale realization functor
\[
\rmR_{\pet}\colon \MT(K)\to \Rep_{\bfQ_p}(\scrG_K);\quad M\mapsto M_{\pet}
\]
induces a natural equivalence between $\MT(K)\otimes\bfQ_p$ and $\Rep_{\bfQ_p}(\scrG_K,\chi_{\cyc,p},\std)$.
For shortening notation, put \[
\MT_p(K):=\Rep_{\bfQ_p}(\scrG_K,\chi_{\cyc,p},\std^{-2}).
\]

Here is the \emph{restricted ramification} version of this category.
Let $\Sigma_K$ be the set of all finite places of $K$ and let $\Sigma$ be a finite subset of $\Sigma_K$.
Then, $\MT_p(\calO_{K,\Sigma})$ denotes the full subcategory of
$\MT_p(K)$ consisting the objects $V$ such that
unramified or crystalline at any $v\in\Sigma_K\setminus \Sigma$ for $v\nmid p$ or $v\mid p$, respectively.
Since the category of crystalline representations forms an abelian category
(\cite{Fo2}), this category is a neutral Tannakian subcategory of $\MT_p(K)$.
When $\Sigma=\emptyset$, $\MT_p(\calO_{K,\emptyset})$ is denoted by$\MT_p(\calO_K)$.
\subsection{Crystalline completion of $\pi_1^{\et}(\scrM_{1,1},v)$}\label{crystallinecomple}
In this subsection, we let the profinite group $\pi$ be $\pi_1^{\et}(\scrM_{1,1},v)$,
define $\rho\colon \pi\to \GL_2(\bfQ_p)$ to be the representation $V_0:=\calV^1_v$,
and take $w\colon \bfG_m\to \GL_2$ to be the diagonal embedding.
Recall that $\scrM_{1,1}$ is the moduli stack of elliptic curves,
$v=\partial/\partial q$ is the standard tangential base point.
The $\bfQ_p$-smooth sheaf $\calV^1$ is defined to be $\rmR^1\pi_*(\bfQ_p)$,
where $\pi\colon \scrE\to \scrM_{1,1}$ is the universal elliptic curve.
Then, the representation $V_0$ of $\pi_1^{\et}(\scrM_{1,1},v)=\widehat{\SL_2(\bfZ)}\rtimes \scrG_{\bfQ}$ is naturally isomorphic to $\bfQ_p\oplus \bfQ_p(-1)$.

By definition, $p$-adic weighted module of $\pi=\pi_1^{\et}(\scrM_{1,1},v)$ with respect to
$(\rho,w)$ is a finite dimensional $\pi$-module $V$ with $W_\bullet V$
such that $\Gr^W_n(V)$ is isomorphic to a direct sum of pull-backs of the $\GL_2$-modules
\[
\Sym^m(\std_{\GL_2})\otimes{\det}^r,\quad m+2r=n
\]
by $\rho$.
The category $\Rep_{\bfQ_p}^{\emptyset}(\pi,\rho,w)$ is defined to be the full subcategory of
$\Rep_{\bfQ_p}(\pi,\rho,w)$ consisting of objects $V$ satisfying the following unramified condition:
\begin{center}
	\textbf{(Ur)} The $\scrG_{\bfQ}$-module $V$ defined to be the pull-back via $\scrG_{\bfQ}\intomap \widehat{\SL_2(\bfZ)}\rtimes\scrG_{\bfQ}= \pi$ is unramified outside $p$
	and crystalline at $p$.\end{center}
The following definition is due to \cite[Subsection 12.1, 12.3]{HM15}.
\begin{definition}The Tannakian fundamental group of $\Rep_{\bfQ_p}^{\emptyset}(\pi,\rho,w)$
	with the base point $\omega_{\mathrm f}$ is called the \emph{crystalline completion of $\pi_1^{\et}(\scrM_{1,1},v)$ with respect to $(\rho,w)$}.
\end{definition}
In this appendix, the symbol $\pi^{\crys}(\rho,w)$ denotes the crystalline completion above.
Similar to the weighted completion, we have a short exact sequence
\[
1\to \pi^{\crys,\unip}(\rho,w)\to \pi^{\crys}(\rho,w)\xrightarrow{\pr} \GL_{2,\bfQ_p}\to 1
\]
of proalgebraic groups over $\bfQ_p$, where $\pi^{\crys,\unip}(\rho,w)$ denotes the prounipotent radical of
$\pi^{\crys}(\rho,w)$.
\begin{proposition}
	\label{propA1}The fiber functor $v$ induces an equivalence
	\[
	\bv \colon \ME_p(\scrM_{1,1})\isom \Rep_{\bfQ_p}^{\emptyset}(\pi,\rho,w)
	\]
	of $\bfQ_p$-linear Tannakian categories.
\end{proposition}
\begin{proof}
	By definition, we have the following diagram
	\[
	\xymatrix{
		\ME_p(\scrM_{1,1})\ar[rr]^{\bv}\ar[d]& & \Rep_{\bfQ_p}^{\emptyset}(\pi,\rho,w)\ar[d]\\
		\mathbf{Sm}_{p}(\scrM_{1,1})\ar[rr]^{\cong}& &\Rep_{\bfQ_p}(\pi_1^{\et}(\scrM_{1,1})),
	}
	\]
	commutative up to natural equivalence.
	As the horizontal functors are fully-faithful, the functor $\bv$ is also fully-faithful.
	
	Let us show the essential surjectivity.
	Let $(V,W_\bullet V)$ be an object of $\Rep_{\bfQ_p}^{\emptyset}(\pi,\rho,w)$.
	Then, as the lower functor in the square above is an equivalence of the category,
	we have a smooth $\bfQ_p$-sheaves $W_\bullet\calF$ whose fibers are isomorphic to $W_\bullet V$.
	Then, since \[
	\Gr^W_n(V)\cong \bigoplus_{m+2r=n}\left(\Sym^m(V_0)\otimes{\det}^{r}\right)^{\oplus i_{n,r}},\quad i_{n,r}\in \bfZ_{\geq 0},
	\]
	we have
	\[
	\Gr^W_n(\calF)\cong \bigoplus_{m+2r=n}\calV^m(-r)^{\oplus i_{n,r}}.
	\]
	Of course, the fiber of $(\calF,W_\bullet\calF)$ at $v$ is isomorphic to $(V,W_\bullet V)$.
	Therefore, $(\calF,W_\bullet\calF)$ is an object of $\ME_p(\scrM_{1,1})$ and we have the conclusion of the proposition.
\end{proof}
The basic properties of the crystalline completion was proved in \cite[Section 13]{HM15}.
Let
\begin{equation}
	\label{eqa1}
	\rmH^i(\pi^{\crys}(\rho,w),\Sym^m(V_0)(r))\to \rmH^i(\pi_1^{\et}(\scrM_{1,1},v),\Sym^m(V_0)(r))
\end{equation}
be a natural homomorphism induced by $\Rep_{\bfQ_p}^{\emptyset}(\pi,\rho,w)\to \Rep_{\bfQ_p}(\pi_1^{\et}(\scrM_{1,1},v))$.
Here, for any $\GL_2$-module $M$ and an integer $r$, $M(r)$ denotes the $\GL_2$-module $M\otimes\det^{\otimes(-r)}$ (of course, this is regarded as a $\pi^{\crys}(\rho,w)$-module).
Let us summarizes their results as follows:
\begin{proposition}[{\cite[Proposition 13.1, Theorem 13.3]{HM15}}]
	\label{propA0}Let $m,r$ be integers such that $m$ is non-negative.
	\begin{itemize}
		\item[(1)]When $i=0$, then the homomorphism (\ref{eqa1}) is always an isomorphism.
		\item[(2)]When $i=1$, (\ref{eqa1}) is isomorphism if $m-2r<-2$.
		If $m-2r\geq -2$, then the source group vanishes.
		\item[(3)]When $i=2$, (\ref{eqa1}) induces an injection
		\[
		\rmH^2(\pi^{\crys}(\rho,w),\Sym^m(V_0)(r))\hookrightarrow \rmH^1_{\mathrm f}(\bfQ,H^1_{\et}(\scrM_{1,1,\overline{\bfQ}},\calV^m(r))).
		\]
	\end{itemize}
\end{proposition}
\subsection{On bisplittings}\label{bisplit}
In this subsection, we show that there is a functorial bigrading on every object of $ \Rep_{\bfQ_p}^{\emptyset}(\pi,\rho,w)$. This fact was previously established in \cite[Appendix B]{HM15}.
We give a slightly different proof of theirs in this subsection.

Let $\rho\colon \pi_1^{\et}(\scrM_{1,1},v)\to \GL_2(\bfQ_p)$ be the same as in the previous subsection. Then, we have the commutative diagram
\[
\xymatrix{
	\pi_1^{\et}(\scrM_{1,1},v)\ar[d]\ar[rr]^{\rho}& &\GL_2(\bfQ_p)\ar[d]^{\det^{-1}}\\
	\scrG_{\bfQ}\ar[rr]^{\chi_{\cyc,p}}& &\bfG_m(\bfQ_p),
}
\]
where the left horizontal morphism is induced one by the structural morphism of $\scrM_{1,1}$ over $\bfQ$.
This defines a natural functor
\[
\iota\colon \MT_p(\bfZ)\to \Rep_{\bfQ_p}^{\emptyset}(\pi,\rho,w)
\]
of Tannakian categories. On the other hand, by defining $c\colon \bfG_m\to \GL_2$ by $c(\sigma)=\begin{pmatrix}
	1&0\\
	0&\sigma^{-1}
\end{pmatrix}$, we have the following commutative diagram:
\[
\xymatrix{
	\pi_1^{\et}(\scrM_{1,1},v)\ar[rr]^{\rho}& &\GL_2(\bfQ_p)\\
	\scrG_{\bfQ}\ar[u]\ar[rr]^{\chi_{\cyc,p}}& &\bfG_m(\bfQ_p)\ar[u]^{c}.
}
\]
Here, the left horizontal homomorphism is defined by the splitting $\pi_1^{\et}(\scrM_{1,1},v)=\widehat{\SL_2(\bfZ)}\rtimes \scrG_{\bfQ}$ induced by $v$.
Therefore, we also have the functor
\[
\bv_*\colon \Rep_{\bfQ_p}^{\emptyset}(\pi,\rho,w)\to \MT_p(\bfZ).
\]
To distinguishes weight filtrations, we use the notations $W_\bullet$ and $M_\bullet$ for weights of object of $\Rep_{\bfQ_p}^{\emptyset}(\pi,\rho,w)$ and $\MT_p(\bfZ)$, respectively.
For each object $V$ of $\Rep_{\bfQ_p}^{\emptyset}(\pi,\rho,w)$, $M_\bullet V$ denotes
$M_{\bullet}\bv_*(V)$ for simplicity.
\begin{lemma}\label{lema-2}
	The two functors \[
	\Rep_{\bfQ_p}^{\emptyset}(\pi,\rho,w)\to \Rep_{\bfQ_p}^{\emptyset}(\pi,\rho,w)^{\mathrm{
			ss}};\quad V\mapsto \Gr^W_\bullet(V)
	\]
	and
	\[
	\Rep_{\bfQ_p}^{\emptyset}(\pi,\rho,w)\to \MT_p(\bfZ)^{\mathrm{ss}};\quad V\mapsto \Gr^M_\bullet(\bv_*(V))
	\]
	are exact functors. Here, the upper subscription $\mathrm{ss}$ means the full-subcategory consisting of semi-simple objects.
\end{lemma}
\begin{proof}The exactness of the first assertion is follows from the general theory of weighted completion.
	See Subsection \ref{generaltheory}. The second exactness follows from the fact above and the exactness of $v_*$.
\end{proof}
Let us define the functor
\[
\Gr^M_\bullet\Gr^W_\bullet\colon \Rep_{\bfQ_p}^{\emptyset}(\pi,\rho,w)\to\Vec_{\bfQ_p}
\]
by
\[
\Gr^M_\bullet\Gr^W_\bullet(V)=\omega_{\mathrm f}\circ \Gr^M_\bullet\circ\bv_*\circ\Gr^W_\bullet(V).
\]
Then, by Lemma \ref{lema-2}, this functor is exact.
As it is easily checked that this functor is compatible with tensor products, we have the following proposition:
\begin{proposition}
	\label{propa3}
	The functor $\Gr^M_\bullet\Gr^W_\bullet$ is a fiber functor of $ \Rep_{\bfQ_p}^{\emptyset}(\pi,\rho,w)$.
\end{proposition}
\begin{remark}
	For each $V\in\Obj( \Rep_{\bfQ_p}^{\emptyset}(\pi,\rho,w))$, we can impose two filtrations $W_\bullet,M_\bullet$ on the underlying $\bfQ_p$-vector space by the obvious way.
	The functor above clearly coincides the correspondence $V\mapsto  \Gr^M_\bullet\Gr^W_\bullet(V)$
	as usual. Note that we have natural isomorphisms
	\[
	\Gr^M_i\Gr^W_j(V)\xleftarrow{\sim}\frac{M_i(V)\cap W_j(V)}{M_{i-1}(V)\cap W_j(V)+M_i(V)\cap W_{j-1}(V)}\isom  \Gr^W_j\Gr^M_i(V),
	\]
	which are functorial in $V$. Therefore, we also can regard $\Gr^W_\bullet\Gr^M_\bullet(V)$ as a fiber functor on $ \Rep_{\bfQ_p}^{\emptyset}(\pi,\rho,w)$.
	This is the functor studied in \cite[Appendix B]{HM15}.
\end{remark}
\begin{proposition}
	\label{propa4}
	There exists an isomorphism
	\[
	\Gr^M_\bullet\Gr^W_\bullet\cong \omega_{\mathrm f}
	\]
	of fiber functors of $\Rep_{\bfQ_p}^{\emptyset}(\pi,\rho,w)$.
\end{proposition}
\begin{proof}
	Put $\omega:=\Gr^M_\bullet\Gr^W_\bullet,\ \calC:=\Rep_{\bfQ_p}^{\emptyset}(\pi,\rho,w)$ and consider the torsors
	\[
	T:=\underline{\Isom}_{\calC}^{\otimes}(\omega,\omega_{\mathrm f}),\quad T^{\mathrm{ss}}:=\underline{\Isom}_{\calC^{\mathrm{
				ss}}}(\omega|_{\calC^{\mathrm{
				ss}}},\omega_{\mathrm f}|_{\calC^{\mathrm{
				ss}}})
	\]
	under $\pi^{\crys}(\rho,w)$ and $\GL_{2,\bfQ_p}$, respectively.
	As the restrictions of $\omega $ and $\omega_{\mathrm{f}}$ to $\calC^{\mathrm{ss}}$ is isomorphic canonically,
	$T^{\mathrm{ss}}$ is isomorphic to $\GL_{2,\bfQ_p}$ canonically.
	
	To show the proposition, it suffices to show the existence of a $\bfQ_p$-valued point of $T$.
	Let \[
	T\to T^{\mathrm{ss}}=\GL_{2,\bfQ_p};\quad \alpha\mapsto \alpha|_{\calC^{ss}}
	\]
	to be the morphism defined by the restriction to $\calC^{\mathrm{ss}}$.
	Since this morphism is compatible with the canonical surjection $\pi^{\crys}(\rho,w)\ontomap \GL_{2, \bfQ_p}$,
	$T(\overline{\bfQ_p})\to T^{\mathrm{ss}}(\overline{\bfQ_p})$ is also surjective.
	Let $t\in T(\overline{\bfQ_p})$ be a lift of the identity matrix $I_2\in \GL_2(\bfQ_p)=T^{\mathrm{ss}}(\bfQ_p)$.
	Then, the continuous one cocycle $c\colon \scrG_{\bfQ_p}\to\pi^{\crys}(\rho,w)(\overline{\bfQ_p})$
	defined by
	\[
	c(\sigma)t=\sigma(t)
	\]
	has a image on the prounipotent group $\pi^{\crys,\unip}(\rho,w)(\overline{\bfQ_p})$.
	Since any prounipotent group can be written as a projective limit of \emph{central extensions},
	any one cocycle of $\scrG_{\bfQ_p}$ valued in $\pi^{\crys,\unip}(\rho,w)(\overline{\bfQ_p})$
	is one coboundary by Hilbert's Satz 90.
	Therefore, we have
	\[
	c(\sigma)=\sigma(u)u^{-1},\quad \forall \sigma\in \scrG_{\bfQ_p}
	\]
	for some $u\in \pi^{\crys,\unip}(\rho,w)(\overline{\bfQ_p})$.
	Then, the point $u^{-1}t$ is $\scrG_{\bfQ_p}$-invariant and hence a $\bfQ_p$-rational point of $T$.
	This rational point defines a desired isomorphism.
\end{proof}

\end{document}